\documentclass[reqno]{amsart}
\usepackage{amssymb}
\usepackage{amsthm}
\usepackage[margin=1in]{geometry}
\usepackage[colorlinks=true, linkcolor=blue]{hyperref}

\newcommand{\mc}{\mathcal}
\newcommand{\R}{\mathbb R}

\newcommand{\Z}{\mathbb Z}
\renewcommand{\H}{\mathbb H}
\renewcommand{\S}{\mathbb S}
\newcommand{\K}{\mathcal K}
\newcommand{\U}{\mathcal U}
\newcommand{\B}{\mathcal B}
\newcommand{\les}{\lesssim}
\newcommand{\dd}{\,d}
\newcommand{\be}{\begin{equation}}
\newcommand{\ee}{\end{equation}}
\newcommand{\lb}{\label}
\newcommand{\ds}{\displaystyle}
\DeclareMathOperator{\Div}{div}
\renewcommand{\div}{\Div}

\DeclareMathOperator{\Rm}{Rm}
\DeclareMathOperator{\Ric}{Ric}
\DeclareMathOperator{\tr}{tr}
\DeclareMathOperator{\Error}{Error}

\newcommand{\ov}{\overline}

\newtheorem{proposition}{Proposition}

\newtheorem{theorem}[proposition]{Theorem}

\title[Dispersive estimates on Riemannian manifolds]{Dispersive estimates  for Schr\"{o}dinger's and wave equations on Riemannian manifolds}
\author{Marius Beceanu}
\begin{document}
\begin{abstract}
This paper proves $L^p$ decay estimates for Schr\"{o}dinger's and wave equations with scalar potentials on three-dimensional Riemannian manifolds.

The main result regards small perturbations of a metric with constant negative sectional curvature. We also prove estimates on $\S^3$, the three-dimensional sphere, and $\H^3$, the three-dimensional hyperbolic space.

Most of the estimates hold for the perturbed Hamiltonian $H=H_0+V$, where $H_0$ is the shifted Laplacian $H_0=-\Delta+\kappa_0$, $\kappa_0$ is the constant (or asymptotic) sectional curvature, and $V$ is a small scalar potential.

The results are based on direct estimates of the wave propagator.

All results hold in three space dimensions. The metric is required to have four derivatives.
\end{abstract}
\maketitle
\tableofcontents
\section{Introduction}
\subsection{Setup}
Consider a three-dimensional Riemannian manifold $(\mc M, g)$ with metric $g$, having at least two (not necessarily continuous) derivatives.

Assume that the topology of the manifold is trivial, meaning that the manifold is simply connected. In the negative sectional curvature and flat cases, further assume that the exponential map $\exp_{x_0}: T_{x_0}\mc M \to \mc M$ is bijective for any $x_0 \in \mc M$.

The Riemannian curvature tensor $\Rm$ is given by
\be\lb{rm_def}
\Rm(X, Y) Z = \nabla_X \nabla_Y Z - \nabla_Y \nabla_X Z - \nabla_{[X, Y]} Z = (\nabla^2_{XY} - \nabla^2_{YX}) Z,
\ee
where $\nabla_X Y$ is the covariant derivative with respect to the Levi-Civita connection, which is torsion-free, and $\nabla^2_{XY} = \nabla_X \nabla_Y - \nabla_{\nabla_X Y}$. For a thorough introduction to Riemannian geometry see \cite{Cal}, \cite{Doc}, or \cite{Lee}.

The Ricci tensor $\Ric$ and scalar curvature $R$ are defined as
\be\lb{ricci_def}
\Ric(u, v)=\sum_{k=1}^3 [\Rm(e_k, u) v] \cdot e_k,\ R = \tr \Ric,
\ee
where $\{e_1, e_2, e_3\}$ form an orthonormal basis for the tangent space $T_x \mc M$.

Sectional curvature in the plane determined by two linearly independent vectors $v_1, v_2 \in T_x \mc M$ is
$$
\kappa(v_1, v_2) = \frac {[\Rm(v_1, v_2) v_2] \cdot v_1}{|v_1 \wedge v_2|^2},
$$
where $|v_1 \wedge v_2|^2 = |v_1|^2 |v_2|^2 - |v_1 \cdot v_2|^2$. Sectional curvature depends only on the plane, not on the choice of linearly independent vectors in it.

For manifolds with constant sectional curvature $\kappa_0$, the curvature tensor has the special form
\be\lb{rm0}
\Rm_0(v_1, v_2) v_3 = \kappa_0 [(v_2 \cdot v_3) v_1 - (v_1 \cdot v_3) v_2],
\ee

Let $\Gamma$ denote the set of all geodesics $\gamma$ on the manifold $\mc M$ and let $L^1(\Gamma)$ be the Banach space
\be\lb{l1gamma}
L^1(\Gamma)=\{f \mid \sup_{\gamma \in \Gamma} \|f\|_{L^1(\gamma)} < \infty\}.
\ee
We assume that the metric on $\mc M$ is a small perturbation of one with constant sectional curvature $\kappa_0<0$. Let
$$
\kappa_0=-\alpha_0^2<0,\ \Rm - \Rm_0 = \Rm_1.
$$

Consider the wave and Schr\"{o}dinger equations on $\mc M$, with Hamiltonian
\be\lb{H}
H=H_0+V=-\Delta+\kappa_0+V.
\ee
Here $-\Delta$ is the Laplace--Beltrami operator $-\Delta=-\Delta_g = -\div \nabla$ for the metric $g$, while $\kappa_0$ is the constant sectional curvature, in the case of a manifold with constant sectional curvature. The operator $-\Delta+\kappa_0$ is called the shifted Laplacian.

In addition, the Hamiltonian may include a small scalar potential $V$, $H=H_0+V=-\Delta+V$. To keep the length of the paper reasonable, the possibility of adding a large scalar potential $V$ or a first-order magnetic potential $A \nabla$, as in \cite{MaMeTa} or \cite{BecKwo}, will be considered in a future paper. A method for passing from small to large scalar potentials is given in \cite{BecGol1} and \cite{BecGol2}.

The equations studied in this paper are
\be\lb{wave}
\partial^2_t u + H u = F,\ u(0)=u_0,\ \partial_t u(0)=u_1 \text{ (wave)}
\ee
with propagators $\cos(t \sqrt H)$ and $\ds\frac {\sin(t \sqrt H)}{\sqrt H}$ and
\be\lb{sch}
i \partial_t u + H u = F,\ u(0)=u_0 \text{ (Schr\"{o}dinger)}
\ee
with propagator $e^{itH}$. The case when $V=0$ and $H$ is replaced by $H_0=-\Delta+\kappa_0$ is called the free wave (or Schr\"{o}dinger) equation.

The $\kappa_0$ term makes no difference for the Schr\"{o}dinger equation, because it can be modulated away through multiplying the solution by a time-dependent complex phase.

Consider the global Kato class $\K$, introduced in \cite{RodSch} for the Euclidean space $\R^3$,
\be\lb{kato}
\K = \{V \in L^1_{loc} \mid \sup_{x_0 \in \mc M} \int_{\mc M} \frac {|V(x)| \dd x}{j(d(x_0, x))} < \infty\},
\ee
where $\ds j(r) = \frac {\sinh(r \alpha_0)}{\alpha_0}$, together with the modified Kato space $\tilde \K = \K \cap L^1$,
\be\lb{modk}
\tilde \K = \{V \mid \sup_{x_0 \in \H^3} \int_{\H^3} \frac {|V(x)|}{\min(1, d(x_0, x))} \dd x < \infty\}.
\ee

As we shall see, in the negative curvature setting $L^{2, 1} \subset \K$ and $\K^* \subset L^{2, \infty}$. For definitions of the Lorentz spaces $L^{p, q}$ and more, see \cite{BerLof}. Furthermore, $L^1 \cap L^1(\Gamma) \subset \K$.

\subsection{Results} The following is the main result of this paper.
\begin{theorem}\lb{thm1} Consider a three-dimensional Riemannian manifold $(\mc M, g)$ such that its Riemann curvature tensor $\Rm$ satisfies, for $\kappa_0<0$ and $\Rm_0$ given by (\ref{rm0}),
$$
\|\Rm-\Rm_0\|_{L^1(\Gamma)}=\|\Rm_1\|_{L^1(\Gamma)} < \epsilon,\ \|\nabla \Rm\|_{L^1(\Gamma)} < \epsilon,\ \|\nabla^2 \Rm\|_{L^1(\Gamma)} < \epsilon
$$
and in addition
$$
\|\Rm_1\|_{L^1} < \epsilon,\ \|\nabla \Rm\|_{L^1} < \epsilon,\ \|\nabla^2 \Rm\|_{L^1} < \epsilon.
$$
Let $H=-\Delta+\kappa_0+V$ be a Hamiltonian such that $\|V\|_\K<\epsilon$.

Then for sufficiently small $\epsilon>0$ the cosine propagator for the wave equation (\ref{wave}) fulfills the $L^1 \to L^\infty$ bound
\be\lb{new3}
\|\cos(t\sqrt{H}) f\|_{L^\infty} \les t^{-1} \|f\|_{W^{2, 1}}.
\ee
In addition, solutions to Schr\"{o}dinger's equation (\ref{sch}) fulfill the $L^1 \to L^\infty$ pointwise decay estimate
$$
\|e^{itH} f\|_{L^\infty} \les |t|^{-3/2} \|f\|_{L^1}.
$$
For any $\delta \in (0, \alpha_0)$, supposing that $V \in \tilde \K$ and that $\epsilon$ and $\|V\|_{\tilde \K}$ are sufficiently small, one can replace $t$ by $\sinh(t(\alpha_0-\delta))$ in (\ref{new3}).
\end{theorem}
All the conditions stated in terms of the Riemann curvature tensor (\ref{rm_def}) can be stated equivalently in terms of the Ricci tensor (\ref{ricci_def}).

We also include some model results that hold for the manifolds with constant sectional curvature $\S^3$ (the three-dimensional sphere) and $\H^3$ (the three-dimensional hyperbolic space). For simplicity, normalize the sectional curvature to $\pm 1$.

\begin{theorem}\lb{thm2} Solutions to the free wave equation (\ref{wave}) on $\S^3$ satisfy the following estimates:
$$
\bigg\|\frac{\sin(t \sqrt{H_0})}{\sqrt{H_0}}f \bigg\|_{L^\infty} \les \frac 1 {\sin t} \|f\|_{W^{1, 1}},\ \|\cos(t\sqrt{H_0}) f\|_{L^\infty} \les \frac 1 {\sin t} \|f\|_{W^{2, 1}},
$$
and for $p \in (1, 2]$
$$
\| e^{it\sqrt{H_0}} f \|_{L^{p'}} \les |\sin t|^{\frac 1 {p'} - \frac 1 p} \|f\|_{W^{2(\frac 1 p - \frac 1 {p'}), p}}.
$$
\end{theorem}

\begin{theorem} Consider the wave equation (\ref{wave}) on $\H^3$ with a scalar potential $V \in \K$. If $\|V\|_\K$ is small, then
\be\lb{new1}
\bigg\|\frac{\sin(t \sqrt{H})}{\sqrt{H}}f \bigg\|_{L^\infty} \les |t|^{-1} \|f\|_{W^{1, 1}},\ \|\cos(t\sqrt{H}) f\|_{L^\infty} \les |t|^{-1} \|f\|_{W^{2, 1}},
\ee
and for $p \in (1, 2]$
\be\lb{new2}
\| e^{it\sqrt{H}} f \|_{L^{p'}} \les |t|^{\frac 1 {p'} - \frac 1 p} \|f\|_{W^{2(\frac 1 p - \frac 1 {p'}), p}}.
\ee
In addition, solutions to Schr\"{o}dinger's equation (\ref{sch}) fulfill the estimate
$$
\|e^{itH} f\|_{L^\infty} \les |t|^{-3/2} \|f\|_{L^1}.
$$
Supposing that $V \in \tilde \K$ is small in this norm, one can replace $t$ by $\sinh t$ in (\ref{new1}) and (\ref{new2}).
\end{theorem}

\subsection{History of the problem}

Equations (\ref{wave}) and (\ref{sch}) on the flat Euclidean space $\R^n $ have been extensively studied elsewhere; it would be impossible to do justice to the large number of papers written on this topic.

Much work has been done on the wave equation on $\R^n$ with variable coefficients, or, equivalently, for small perturbations of the flat metric on $\R^n$. This includes \cite{GebTat}, \cite{HasZha}, \cite{Kap}, \cite{MaMeTa}, \cite{MetTat}, \cite{MoSeSo}, \cite{ScSoSt1}, \cite{ScSoSt2}, \cite{Smi}, \cite{SmiSog1}, \cite{SmiSog2}, \cite{SogWan}, \cite{StaTat}, \cite{Tat1}, \cite{Tat2}, \cite{Tat3}, \cite{Tat4}, \cite{Tat5}, \cite{Tat6}, and \cite{Zha}, among other papers.

Previous works on decay and Strichartz estimates for the wave (\ref{wave}) and Schr\"{o}dinger (\ref{sch}) equations on hyperbolic space (with constant negative sectional curvature) include \cite{AnkPie1}, \cite{AnkPie2}, \cite{AnPiVa1}, \cite{AnPiVa2}, \cite{Has}, \cite{Ion}, \cite{MetTay}, and \cite{Tat3}.

Most of these works study the shifted wave equation (\ref{wave}), while others, such as \cite{MetTay}, study it without the correction $\kappa_0$.  The $\kappa_0$ term in (\ref{H}) arises naturally from computations and is essential in the proof.

This adjustment also works in the case of small perturbations of a metric with constant sectional curvature $\kappa_0<0$.

Previously, Strichartz estimates for the Schr\"{o}dinger (\ref{sch}) and wave (\ref{wave}) equations on asymptotically hyperbolic non-compact manifolds were proved in \cite{Che} and \cite{SSWZ}.

By comparison, the current work's assumptions are stronger and also produce stronger estimates: $L^p$ decay estimates imply Strichartz estimates. Some estimates are new even in the constant curvature cases $\H^3$ and $\S^3$.

In three dimensions the spectrum of the Laplace-Beltrami operator $-\Delta$ on $\H^3$ is $[-\kappa_0, \infty)$. Thus the spectrum of the shifted Laplacian $-\Delta+\kappa_0$ starts at $0$, so it is in some sense more natural.

Likewise, on $\S^3$ the spectrum of $-\Delta$ is discrete and consists of
$$
\sigma(-\Delta_{\S^3}) = \{\ell(\ell+2)\kappa_0 \mid \ell \in \Z, \ell \geq 0\},
$$
with multiplicity $(\ell+1)^2$. Replacing the Laplacian by the shifted Laplacian (\ref{H}), the spectrum of $-\Delta+\kappa_0$ is now a sequence of perfect squares $(\ell+1)^2 \kappa_0$, which is again more natural.

Due to the finite speed of propagation of solutions to the wave equation, on compact manifolds such as $\S^3$ decay estimates can follow by piecing together local in time estimates on smaller portions, such as those in \cite{Kap} or \cite{MoSeSo}. Thus, Theorem \ref{thm2} is stated more for the simplicity of its proof.

The current paper uses a more geometric method for the study of equations (\ref{wave}) and (\ref{sch}) on curved metric backgrounds, based on $L^1$ estimates, unlike most previous papers that employ local energy decay and $L^2$-based methods.

Future papers will consider other cases, such as that of non-negative sectional curvature or that of four-dimensional pseudo-Riemannian manifolds, in particular Einstein manifolds. These other cases come with an added difficulty, caused by the presence of conjugate points, i.e.~points at which solutions of Jacobi's equation (\ref{jacobi1}) vanish.


\section{Preliminaries}
\subsection{Curvature and Jacobi fields}
Einstein manifolds are those manifolds for which the Ricci tensor is a scalar multiple of the metric:
\be\lb{einstein}
\Ric=\kappa(x) g.
\ee
Note that $g=I$ when written in an orthonormal basis. Consequently, sectional curvature is the same in all directions, at any given point $x$, and equals $\kappa(x)$.

Equation (\ref{einstein}) implies by Schur's lemma, in any dimension greater than two, that $\kappa(x)$ is constant. Thus constant sectional curvature pointwise is the same as constant sectional curvature globally, in dimensions three and higher.

In three space dimensions, the Einstein manifolds are those with constant sectional curvature, which admit as universal cover the sphere $\S^3$, $\R^3$, or the hyperbolic space $\H^3$.

For simplicity, this paper only considers simply connected manifolds $\mc M$ such as $\S^3$, $\R^3$, and the hyperbolic space $\H^3$, as well as small metric perturbations of the latter.

Fix an arbitrary point $x_0 \in \mc M$ and consider metric balls and spheres
$$
B(x_0, r)=\{x \in \mc M \mid d(x_0, x) < r\},\ \partial B(x_0, r)=\{x \in \mc M \mid d(x_0, x) = r\},
$$
as well as the geodesic distance function $r:\mc M \to \R$, $r=d(x_0, x)$, for any $x \in \mc M$. Clearly $|\nabla r|=1$ and $\nabla r$ is normal to $\partial B(x_0, r)$.

Consider any other point $x \in \mc M$, $x \ne x_0$. Let $\{e_1, e_2, e_3\}$ be an orthonormal basis for $T_{x_0} \mc M$, the tangent space of $\mc M$ at $x_0$, with $e_3$ being the tangent vector of the geodesic $x_0x$. By parallel transport we extend $e_1$, $e_2$, and $e_3$ to vector fields along the geodesic $x_0x$, with $e_3=\nabla r$ and $\partial_r = \nabla_{e_3}$.

Note that $e_3 = \omega$ depends on $\omega \in \S^2$, the direction of the geodesic, hence so do $e_1$ and $e_2$. For our computations, $e_1$ and $e_2$ need not be defined for every $\omega \in \S^2$, but only on some neighborhood of the original $\omega$.

For concreteness, consider spherical coordinates on $\S^2 \subset T_{x_0} \mc M$. In the direction $\omega=(\theta, \phi) \in \S^2$ corresponding to neither pole let
$$\begin{aligned}
\omega=e_3&=(\sin \theta, \cos \theta \cos \phi, \cos \theta \sin \phi) \\
e_1&=(\cos \theta, -\sin \theta \cos \phi, -\sin \theta \sin \phi) \\
e_2&=(0, -\sin \phi, \cos \phi).
\end{aligned}$$
This assignment is discontinuous at the poles. In some sense this is the best we can do, since the tangent bundle of $\S^2$ is not parallelizable, but $\S^2$ can be covered by finitely many (two) such charts.

To understand the area element $dA$ of $\partial B(x_0, r)$ at $x$ as a function of $r$, consider two Jacobi vector fields $J_1(r)$ and $J_2(r)$ defined along the geodesic $x_0x$ as solutions of the Jacobi equation
\be\lb{jacobi1}
\partial^2_r J + \Rm(J, e_3) e_3 = 0,
\ee
with initial conditions
\be\lb{jacobi2}
J_1(0)=J_2(0)=0,\ \partial_r J_1(0) = e_1,\ \partial_r J_2(0) = e_2.
\ee
The vector fields $J_k$ generalize the generators of rotations in symmetric spaces.

Then
$$
|dA(r)| = |J_1(r) \wedge J_2(r)| \dd \omega.
$$
Since $J_k(0)$ and $\partial_r J_k(0)$ are orthogonal to $e_3$, the same is true for all $r$ (see for example \cite{Lee}), so $J_1$ and $J_2$ are always in the plane spanned by $e_1$ and $e_2$.

Let $T$ be the linear transformation that takes $e_k$ to $J_k$, $1 \leq k \leq 2$, in the subspace of the tangent space $T_x \mc M$ orthogonal to $e_3$. Written in coordinates, in the basis $\{e_1, e_2\}$, this is a $2 \times 2$ matrix
\be\lb{t}
T = \begin{pmatrix} j_{11} & j_{21} \\ j_{12} & j_{22} \end{pmatrix},\ j_{k\ell} = J_k \cdot e_\ell, 1 \leq k, \ell \leq 2,
\ee
such that $J_1 \wedge J_2 = \det T$.

Since $T$ is scalar, $\nabla_{e_k} T$, $\nabla_{J_k} T$, or $\partial_r T$ mean the action of vector fields on (the entries of) $T$.

Let $a=J_1 \wedge J_2 = \det T$; then the area element of $\partial B(x_0, r)$ is $dA=a \dd\omega$. Any unimodular basis change applied to $\{e_1, e_2\}$ leads to the same for $\{J_1, J_2\}$ and to conjugation for the matrix $T$. In particular, $a=J_1 \wedge J_2$ does not depend on any particular choice of orthonormal basis for the initial conditions (\ref{jacobi2}).

Write both the Riemann tensor and the Ricci tensor in coordinates,
$$
\Rm_{k\ell mn} = \Rm(e_k, e_\ell, e_m, e_n) = [\Rm(e_k, e_\ell) e_m] \cdot e_n,\ \Ric_{k \ell} = \Ric(e_k, e_\ell),
$$
where $\{e_1, e_2, e_3\}$ is the orthonormal basis defined above by parallel transport.

In the constant curvature case, when written in coordinates using an orthonormal basis
$$
\Ric_0 = 2\kappa_0 g = 2\kappa_0 I,\ R_0 = 6 \kappa_0.
$$

In three space dimensions the Ricci tensor completely determines the Riemann curvature tensor. Condition (\ref{rie}) and others can be equivalently expressed in terms of the Ricci tensor, for example
$$
\int_\R |\Ric(\gamma(s)) - \Ric_0| \dd s \les \epsilon << 1,
$$
up to a constant factor.

Let $\mc A$ be the linear transformation
$$
\mc A J = \Rm(J, e_3) e_3
$$
that appears in Jacobi's equation (\ref{jacobi1}). Written in coordinates, $\mc A$ is the $2 \times 2$ symmetric matrix
$$
\mc A_{k \ell}=(\Rm_{k33\ell})_{1 \leq k, \ell \leq 2}.
$$
Then the equation of $T$ is
\be\lb{eqt}
\partial^2_r T + \mc A(r, \omega) T = 0, T(r=0)=0, \partial_r T(r=0) = I.
\ee
If we likewise set $\mc A_0 = (\Rm_0)_{k33\ell}$, then $\mc A_0=\kappa_0 I$ and $|\mc A-\mc A_0| \leq |\Rm-\Rm_0|$ (or $|\mc A-\mc A_0| \les |\Ric-\Ric_0|$, where the constant is $3/2$). For convenience, let $\mc A_1=\mc A  - \mc A_0$.

For $k \ne \ell$, $\mc A_{k \ell} = \Ric_{k \ell}$, while
$$
\mc A_{11}=\frac 1 2(\Ric_{11}+\Ric_{33}-\Ric_{22}) = \frac R 2 - \Ric_{22}
$$
and likewise for $\mc A_{22}$. In particular, $\tr \mc A = \Ric_{33}$.

By this and similar algebraic identities, it follows that, pointwise, sectional curvature is uniformly close to some value $\kappa(x)$ for all tangent planes in $T_x \mc M$ if and only if the Ricci tensor is close to a scalar multiple of the identity (i.e.~of the metric) $2\kappa(x) I = 2 \kappa(x) g$. Thus, all possible ways in which we can state condition (\ref{rie}) are equivalent.


\subsection{Notations}
This paper employs the following notations:
\begin{itemize}
\item $a \les b$ means $|a| \leq C |b|$ for some constant $C$
\item $C^\infty_c(D)=\mc D(D)$ is the set of test functions (smooth compactly supported functions) on $D$
\item $L^p$ are Lebesgue spaces, $\dot W^{s, p}$ are homogeneous Sobolev spaces
\item $\delta_{x_0}(x)$ is the Dirac delta at $x_0$, a pure point measure
\item $\K$ is the Kato space (\ref{kato}) and $\tilde \K$ is the modified Kato space (\ref{modk})
\item $\mc M$ is a Riemannian manifold
\item $\S^d$ is the $d$-dimensional sphere, $\H^d$ is the $d$-dimensional hyperbolic space
\item $\Rm$ is the Riemann curvature tensor and $\Ric$ is the Ricci tensor
\item $\Gamma$ is the set of all geodesics on the manifold $\mc M$
\item $x_0$ is an arbitrary fixed point on $\mc M$ and $r$ is the geodesic distance from $x$ to $x_0$
\item $L^1(\Gamma)=\{f \mid \sup_{\gamma \in \Gamma} \|f\|_{L^1(\gamma)} < \infty\}$
\end{itemize}

\section{Constant sectional curvature}
\subsection{The sine propagator}
By Huygens's principle, in odd space dimensions $\R^d$, except $d=1$, the propagator for the free wave equation is supported on the surface of a sphere of radius $|t|$ around the source at time $t$. The sine propagator in $\R^3$ is given by
$$
\frac {\sin(t\sqrt{-\Delta})}{\sqrt{-\Delta}} = \frac 1 {4\pi t} \delta_{|x-y|=|t|}.
$$
Surprisingly perhaps, the same is true on $\S^3$ and $\H^3$, following appropriate adjustments.
\begin{proposition}\lb{prop1} Consider a simply connected three-dimensional Riemannian manifold $(\mc M, g)$ of constant sectional curvature $\kappa_0$. Then the sine propagator for the free wave equation (\ref{wave}) is
$$
\frac {\sin(t \sqrt{-\Delta+\kappa_0})}{\sqrt{-\Delta+\kappa_0}}(x_0, x) = \frac 1 {4\pi j(t)} \delta_{d(x_0, x)=|t|},
$$
with $j$ given by (\ref{j}).
\end{proposition}
\begin{proof}
Consider a simply connected three-dimensional Riemannian manifold $(\mc M, g)$, which for now is either $\S^3$, $\R^3$, or $\H^3$. On this manifold, consider solutions of the wave equation (\ref{wave}) of the form
\be\lb{ff}
u(x, t) = f(r) \chi(r-t)
\ee
where $\chi$ is any smooth cutoff function and $f$ depends on the metric. In the flat metric case, $\ds f(r)=\frac 1 {r}$.


Let $x_0 \in \mc M$ be an arbitrary point on the manifold and let $r=d(x_0, x)$ be the geodesic distance from $x_0$. Recall $|\nabla r|=1$ and $\nabla r$ is normal to $\partial B(x_0, r)$. Then
$$
\partial^2_t u = f(r) \chi''(r-t), \nabla u = (f'(r) \chi(r-t) + f(r) \chi'(r-t)) \nabla r,
$$
and
\be\lb{delta1}
\Delta u = (f'(r) \chi(r-t) + f(r) \chi'(r-t)) \Delta r + [f''(r) \chi(r-t) + 2 f'(r) \chi'(r-t) + f(r) \chi''(r-t)]. 
\ee
In order for $u$ to be a solution of the wave equation (\ref{wave}), all terms except $f(r) \chi''(r-t)$ must cancel (or add up to a scalar multiple of $u$, see below).

Note that $\Delta r = \div \nabla r$ is the infinitesimal amount by which the area element of $\partial B(x_0, r)$ changes:
\be\lb{divergence}
\Delta r = \div \nabla r = \frac {\partial_r dA}{dA}.
\ee


For manifolds with constant sectional curvature, the curvature tensor, $\Rm_0$, is given by (\ref{rm0}), so
$$
\Rm_0(J, e_3) e_3 = \kappa_0 [|e_3|^2 J - (J \cdot e_3) e_3] = \kappa_0 J.
$$
Therefore the Jacobi equation has a particularly simple form, $\partial^2_r J + \kappa_0 J = 0$, and
in all three situations $J_k(r)=j(r)e_k$, where $j$ is the solution of the scalar Jacobi equation
$$
j''+\kappa_0 j =0,\ j(0)=0,\ j'(0)=1.
$$
In fact
\be\lb{j}
j(r) = \left\{\begin{aligned}
&\frac {\sin(r \sqrt{\kappa_0})}{\sqrt {\kappa_0}}, &&\kappa_0>0,\\
&\frac {\sinh(r \sqrt {-\kappa_0})}{\sqrt{-\kappa_0}}, &&\kappa_0<0,\\
&r, &&\kappa_0=0.
\end{aligned}\right.
\ee
Hence the area element of $\partial B(x_0, r)$ is $dA=a \dd \omega$, where $a=j^2(r)$,
and by (\ref{divergence})
$$
\Delta r = \frac {2j'(r)}{j(r)}.
$$
%
The result does not depend on $\omega \in \S^2 \subset T_{x_0} \mc M$, the direction of the geodesic.

We now impose the condition $f(r) \Delta r + 2f'(r) = 0$, so that the coefficient of $\chi'(r-t)$ in (\ref{delta1}) is $0$. Then
$$
\frac {f'(r)}{f(r)} = -\frac 1 2 \Delta r = -\frac {j'(r)}{j(r)},
$$
hence
$$
f(r)=\frac c {j(r)}.
$$
Choose $c=1$ and $\ds f(r)=\frac 1 {j(r)}$. Then the coefficient of $\chi(r-t)$ in (\ref{delta1}) is
$$
f'(r) \Delta r + f''(r) = - \frac {j'}{j^2} \frac {2j'} j + \frac {2(j')^2}{j^3} - \frac {j''}{j^2} = - \frac {j''}{j^2} = \frac {\kappa_0} j = \kappa_0 f.
$$
Consequently, for this choice of $f$ and arbitrary $\chi$, $u(x, t)=f(r)\chi(r-t)$ (\ref{ff}) solves the free modified wave equation
\be\lb{freemod}
\partial^2_t u - (\Delta+\kappa_0) u = 0.
\ee

Next, we construct a family of approximations for the sine propagator. Choose $\chi \in C^\infty_c(0, \infty)$ such that $\chi(0)=0$, $\chi(r)=0$ for $r \geq 2$, and $\int_0^\infty \chi(r) \dd r = 1$.

Consider the solutions $u_\delta=f(r)\chi_\delta(r-t)$ corresponding to $\chi_\delta(r)=\delta^{-1} \chi(r/\delta)$ and $\delta>0$. Letting $\delta$ approach $0$, the initial data $(u_\delta(0), \partial_t u_\delta(0))$ are supported inside a small geodesic ball $B(x_0, 2\delta)$, which is diffeomorphic to a ball in the tangent space $T_{x_0}\mc M$, with the diffeomorphism (the exponential map) approaching an isometry as $\delta \to 0$.

Hence as $\delta$ approaches $0$, $u_\delta(0)$ approaches $0$ in $L^1$ since
$$
\int_{\R^3} \delta^{-1} \chi(r/\delta) \dd x = \delta^{-1} \int_0^\infty \chi(r/\delta) r^2 \dd r = O(\delta^2)
$$
and $\ds \frac 1 {4\pi} \partial_t u_\delta(0) = -\frac 1 {4\pi} \delta^{-2} \chi'(r/\delta)/r$ is an approximation for $\delta_{x_0}$. Indeed
\be\lb{approx}
-\frac 1 {4\pi} \int_{\R^3} \delta^{-2} \chi'(r/\delta)/r \dd x = -\delta^{-2} \int_0^\infty \chi'(r/\delta) r \dd r = -\int_0^\infty \chi'(r) r \dd r = \int_0^\infty \chi(r) \dd r = 1
\ee
implies that $\ds \frac 1 {4\pi} \partial_t u_\delta(0)$ is an approximation to the identity as $\delta \to 0$.

In the limit, the initial data $\ds \frac 1 {4\pi}(u_\delta(0), \partial_t u_\delta(0))$ approach $(0, \delta_{x_0})$. Since $\delta^{-1}\chi((r-t)/\delta) \to \delta_r(t)$ when $\delta \to 0$, for $t>0$ the sine propagator for the free modified wave equation (\ref{freemod}) is given by
$$
S(t)(x, x_0) = \frac 1 {4\pi} f(t) \delta_{d(x_0, x)=t} = \frac 1 {4\pi j(t)} \delta_{d(x_0, x)=t}
$$
and this extends by parity to $t<0$.
\end{proof}

\subsection{Other computations for constant sectional curvature}
These computations are for future reference. We shall redo them more generally in the non-constant scalar curvature case, where we know to expect similar results.

We compute $\nabla_{e_k} e_\ell$ for $1 \leq k, \ell \leq 3$. By parallel transport $\nabla_{e_3} e_\ell = 0$ for $1 \leq \ell \leq 3$. Using the definition of the Jacobi fields with families of geodesics, it follows that $[e_3, J_k]=0$ for $1 \leq k \leq 2$, so $\nabla_{J_k} e_3 = \nabla_{e_3} J_k = j'(r) e_k$. Consequently $\ds\nabla_{e_k} e_3 = \frac {j'(r)}{j(r)} e_k$ and
$$
[e_k, e_3] = \nabla_{e_k} e_3 - \nabla_{e_3} e_k = \frac {j'(r)}{j(r)} e_k
$$
as well.

By the definition of curvature, since $[e_3, J_k]=0$
$$
\nabla_{e_3} \nabla_{J_k} e_\ell - \nabla_{J_k} \nabla_{e_3} e_k = \Rm(e_3, J_k) e_\ell,
$$
which becomes, for constant sectional curvature,
$$
\partial_r [\nabla_{J_k} e_\ell] = -\delta_k^\ell j''(r) e_3
$$
so, since in fact $[\nabla_{J_k} e_\ell](x_0)=-\delta_k^\ell e_3$,
$$
\nabla_{e_k} e_\ell = -\delta_k^\ell \frac {j'(r)}{j(r)} e_3.
$$
Note that $\nabla_{e_k} J_\ell = \nabla_{J_k} e_\ell = -\delta_k^\ell j'(r) e_3$
for $1 \leq k \leq 2$.

We need the magnitude of $J_1 \wedge J_2$, for which $\nabla_{e_k} (J_1 \wedge J_2) = 0$ (it is constant on $\partial B(x_0, r)$), not the orientation. The magnitude is determined by the matrix $T$
(\ref{t}):
$$
J_1 \wedge J_2 = (J_1 \cdot e_1) (J_2 \cdot e_2) - (J_1 \cdot e_2) (J_2 \cdot e_1) = j_{11} j_{22} - j_{12} j_{21} =  \det T,
$$
where $\nabla_{e_k} T = 0$, for $1 \leq k \leq 2$, since $T$ is constant on $\partial B(x_0, r)$.


Also
$$
\div e_1 = (\nabla_{e_1} e_1) \cdot e_1 + (\nabla_{e_2} e_1) \cdot e_2 + (
\nabla_{e_3} e_1) \cdot e_3 = 0
$$
and same for $e_2$, while
$$
\div e_3 = (\nabla_{e_1} e_3) \cdot e_1 + (\nabla_{e_2} e_3) \cdot e_2 + (\nabla_{e_3} e_3) \cdot e_3 = \frac {2j'(r)}{j(r)}.
$$

Finally, for $1 \leq k, \ell \leq 2$,
$$
\ds\nabla^2_{e_k} e_\ell = \nabla_{e_k} (\nabla_{e_k} e_\ell) - \nabla_{\nabla_{e_k} e_k} e_\ell = \nabla_{e_k} (-\delta_k^\ell \frac {j'(r)}{j(r)} e_3) - \nabla_{-\frac {j'(r)}{j(r)} e_3} e_\ell= - \delta_k^\ell \frac {(j'(r))^2}{j^2(r)} e_k
$$
and
$$
\ds\nabla^2_{e_k} e_3 = \nabla_{e_k} (\nabla_{e_k} e_3) - \nabla_{\nabla_{e_k} e_k} e_3 = \nabla_{e_k} (\frac {j'(r)}{j(r)} e_k) - \nabla_{-\frac {j'(r)}{j(r)} e_3} e_3 = - \frac {(j'(r))^2}{j^2(r)} e_3.
$$

\section{Non-constant negative sectional curvature}
\subsection{Integrability along geodesics}
Consider a three-dimensional Riemannian manifold $\mc M$. Assume that the metric is a small perturbation of a constant negative sectional curvature metric, in the sense that along any geodesic $\gamma$ on $\mc M$
\be\lb{rie}
\|\Rm_1\|_{L^1(\gamma)} = \|\Rm-\Rm_0\|_{L^1(\gamma)} = \int_\R |\Rm_{\gamma(s)}-\Rm_0| \dd s < \epsilon << 1,
\ee
in other words $\|\Rm-\Rm_0\|_{L^1(\Gamma)} < \epsilon$, where
$$
\Rm_0(v_1, v_2)v_3 = \kappa_0[(v_2 \cdot v_3) v_1 - (v_1 \cdot v_3) v_2]
$$
is the Riemann curvature tensor corresponding to constant sectional curvature $\kappa_0 < 0$, see (\ref{rm0}). Let
$$
\mc A = (\Rm_{k33\ell})_{1 \leq k, \ell \leq 2},\ \mc A_0 =  \kappa_0 I,\ \mc A - \mc A_0 = \mc A_1.
$$

We keep the following notations: an arbitrary $x_0 \in \mc M$ is the origin, $r=d(x_0, x)$, $\{e_1, e_2, e_3\}$ is an orthonormal basis first defined for $T_{x_0} \mc M$, such that $e_3$ is tangent to the geodesic $x_0 x$, then extended by parallel transport along each geodesic $x_0 x$, starting at $x_0$ and reaching to any $x \in \mc M$.

It is not necessary for $e_1$ and $e_2$ to be defined at once for all $\omega \in \S^2$; it suffices to define them on a finite number of patches that constitute a covering of $\S^2$.

The vector fields $J_k$, $1 \leq k \leq 2$, are the two Jacobi vector fields defined in (\ref{jacobi1}) and (\ref{jacobi2}). For comparison purposes, we also retain the fundamental solution
\be\lb{jneg}
j(r) = \frac {\sinh(r \alpha_0)}{\alpha_0}
\ee
from the constant curvature case (\ref{j}), where recall $\alpha_0=\sqrt{-\kappa_0}$.

The first result is an example of what can be achieved in this setting. In addition to (\ref{rie}), we also assume that
\be\lb{rie1}
\|\nabla \Rm\|_{L^1(\Gamma)} = \sup_{\gamma \in \Gamma} \int_\R |\nabla \Rm_{\gamma(s)}| \dd s < \epsilon,
\ee
where the supremum is taken over all geodesics $\gamma \in \Gamma$ on $\mc M$, and
\be\lb{rie2}
\|\nabla^2 \Rm\|_{L^1(\Gamma)} = \sup_{\gamma \in \Gamma} \int_\R |\nabla^2 \Rm_{\gamma(s)}| \dd s < \epsilon.
\ee

Similar assumptions of integrability along the flow (in their case the Hamilton flow, in the current paper the geodesic flow) also appear in \cite{MaMeTa1} and \cite{MetTat}.

Assumptions (\ref{rie}) and (\ref{rie1}) imply that sectional curvature on the manifold $\mc M$ is negative.

\begin{proposition}\lb{prop2}
Consider a simply connected three-dimensional Riemannian manifold $(\mc M, g)$ and suppose that the curvature tensor $\Rm$ satisfies (\ref{rie}), (\ref{rie1}), and (\ref{rie2}). If $\epsilon<<1$, then the sine propagator for the modified wave equation (\ref{wave}) fulfills the estimate
$$
\sup_{x_0 \in \mc M} \bigg\|\frac {\sin(t \sqrt {-\Delta+\kappa_0})}{\sqrt{-\Delta+\kappa_0}}(x, x_0)\bigg\|_{L^1_x} \les e^{c\epsilon t} j(t).
$$
\end{proposition}
The sine propagator is supported on the ball $B(x_0, t)$ at time $t>0$. The volume of $B(x_0, r)$ is of size, up to a constant factor,
$$
\int_0^r j^2(s) \dd s \sim \sinh(2\alpha_0 r) - 2\alpha_0 r.
$$
Thus, the above estimate implies the propagator's decay in average, as its average over $B(x_0, t)$ is of size roughly $\ds\frac {e^{c\epsilon t}}{j(t)} \to 0$ as $t \to \infty$.
\begin{proof}
Fix $x_0 \in \mc M$ and consider solutions of (\ref{wave}) the form
\be\lb{fx}
u(x, t)=f(x) \chi(r-t),
\ee
where $r=d(x_0, x)$. When sectional curvature is not constant, one cannot choose $f$ to be radially symmetric, unlike in (\ref{ff}). Then
$$
\partial_t^2 u = f(x) \chi''(r-t),\ \nabla u=(\nabla f(x))\chi(r-t) + f(x) \chi'(r-t) \nabla r,
$$
and
\be\lb{exx}
\Delta u = (\Delta f(x))\chi(r-t) + 2 (\nabla f(x)) \cdot \chi'(r-t) (\nabla r) + f(x) (\chi''(r-t) + \chi'(r-t) \Delta r).
\ee
Here $\nabla f \cdot \nabla r = \partial_r f$. We impose the condition that
$$
2 \partial_r f + (\Delta r) f = 0,
$$
so the terms involving $\chi'$ cancel in (\ref{exx}). Recall that $\ds \Delta r = \frac {\partial_t dA}{dA}$ by (\ref{divergence}) and (\ref{equiv}), so the previous equation means
\be\lb{f}
f=\frac {c(\omega)} {\ds\sqrt {\frac {dA}{d\omega}}}.
\ee
Using Jacobi fields, $dA=a \dd \omega$, where $a=J_1 \wedge J_2=\det T$, see (\ref{t}). We again take $c(\omega)=1$. However, in this case $dA$ depends on $x$ through not just $r$, but also $\omega$: $dA$ may take different values along different geodesics starting from $x_0$.

Unlike in the constant curvature case, we must keep in mind that $f$ (\ref{f}) only produces an approximate solution, which we can however bootstrap.

We can now prove (\ref{divergence}). Indeed,
\be\lb{equiv}\begin{aligned}
\div e_3 &= (\nabla_{e_1} e_3) \cdot e_1 + (\nabla_{e_2} e_3) \cdot e_2 + (\nabla_{e_3} e_3) \cdot e_3 = (T^{-1} \nabla_{J_1} e_3) \cdot e_1 + (T^{-1} \nabla_{J_2} e_3) \cdot e_2 \\
&=(T^{-1} \nabla_{e_3} J_1) \cdot e_1 + (T^{-1} \nabla_{e_3} J_2) \cdot e_2= (T^{-1} (\partial_r T) e_1) \cdot e_1 + (T^{-1} (\partial_r T) e_2) \cdot e_2 \\
&= \tr [T^{-1} (\partial_r T)] = \frac {\partial_r \det T}{\det T} = \frac {\partial_r a}a.
\end{aligned}\ee
by Jacobi's identity.

The Jacobi equation (\ref{jacobi1}) becomes
\be\lb{jacobi3}
\partial^2_r J - \kappa_0 J = -\Rm_1(J, e_3) e_3 = -\mc A_1 J.
\ee
This is equivalent to
\be\lb{semi}
J(r, \omega) = j'(r) J(0) + j(r) \partial_r J(0) - \int_0^r j(r-s) [\mc A_1 J](s) \dd s.
\ee
Differentiating (\ref{semi}) leads to
\be\lb{cos}
\partial_r J(r, \omega) = j''(r) J(0) + j'(r) \partial_r J(0) - \int_0^r j'(r-s) [\mc A_1 J](s) \dd s.
\ee
In both cases $J_k(0)=0$, so the first term vanishes, by (\ref{jacobi2}).

Note that
$$
0 \leq \sinh(r-s) \sinh s \leq \sinh(r-s) \cosh s = \frac 1 2 (\sinh r + \sinh(r-2s)) \leq \sinh r,
$$
so
$$
\int_0^r j(r-s) [\mc A_1 J](s) \dd s \leq \frac 1 {\alpha_0} j(r) \|\Rm_1\|_{L^1(x_0 x)} \|J\|_{j(r) L^\infty_r}.
$$
This mapping is a contraction when $\|\Rm_1\|_{L^1(x_0x)}<\alpha_0$. If this norm is large, but finite, we can still obtain some bounds by subdividing the geodesic into finitely many segments on which the norm is small, but this is insufficient for the subsequent argument.

Using a contraction argument in the norm $j(r)L^\infty_r$, for sufficiently small $\epsilon$ in (\ref{rie}) we obtain that
\be\lb{first}
|J_k(r, \omega)-j(r)e_k| \les \epsilon j(r),
\ee
which by (\ref{cos}) also implies (note $j$ on the right-hand side, not $j'$)
\be\lb{second}
|\partial_r J_k(r, \omega)-j'(r)e_k| \les \epsilon j(r).
\ee

The same contraction argument applies in general to solutions of Jacobi's equation (\ref{jacobi1}), meaning that the equation has a propagator $k$ such that
\be\lb{propk}
k(t, s) \les j(t-s).
\ee
Due to exponential growth, most of the quantities we compute are of size $e^{\alpha_0 r}$, while quadratic quantities, such as $\det T = J_1 \wedge J_2$, are of size $e^{2\alpha_0 r}$.

An improvement is possible for small $r$: noting that (\ref{semi}) gains two powers of $r$ for small $r$ over the forcing term, while (\ref{cos}) also gains one power, for $r<1$
\be\lb{smallr2}
|J_k(r, \omega) - j(r)e_k| \les \epsilon r^3
\ee
and
\be\lb{smallr}
|\partial_r J_k(r, \omega) - j'(r)e_k| \les \epsilon r^2
\ee
as well.

Due to (\ref{jacobi2}), $T(r=0)=0$ and $\partial_r T(r=0) = I$. Plugging (\ref{smallr}) back into the equation (\ref{jacobi1}), to a second approximation, for $r<1$, $T$ is given by
\be\lb{expand}
T = rI + \frac {r^3} 6 \mc A(x_0, \omega) + o(r^3),
\ee
where $\mc A(x_0, \omega) = (\Rm_{k33\ell})_{1 \leq k, \ell \leq 2}(x_0)$ depends on $e_3=\omega$.

By (\ref{first}) and (\ref{second}),
$$
|T-j(r)I| + |\partial_r T - j'(r) I| \les \epsilon j(r),
$$
which also implies that $\ds\bigg|T^{-1} - \frac 1 {j(r)} I\bigg| \les \frac \epsilon {j(r)}$.

Consequently, for $a=J_1 \wedge J_2 = \det T$,
\be\lb{both}
|a(r, \omega) - j^2(r)| \les \epsilon j^2(r),\ |\partial_r a(r, \omega) - 2 j(r) j'(r)| \les \epsilon j(r) j'(r).
\ee

Estimates (\ref{smallr2}) and (\ref{smallr}) produce an improvement for $r<1$:
\be\lb{improv}
|a(r, \omega) - j^2(r)| \les \epsilon r^4,\ |\partial_r a(r, \omega) - 2j(r)j'(r)| \les \epsilon r^3,\ |(\partial_r a(r, \omega))^2 - 4j^2(r)(j'(r))^2| \les \epsilon r^4.
\ee

We now compute $\nabla_{e_k} e_\ell$, for $1 \leq k, \ell \leq 3$. For some of these expressions we only need an upper bound. First, consider the case when $k$ or $\ell$ is $3$.

Due to parallel transport, $\nabla_{e_3} e_k = 0$. Since $[e_3, J_k] = 0$,
$$
\nabla_{J_k} e_3 = \nabla_{e_3} J_k = (\partial_r T) e_k = j'(r) e_k + O(\epsilon j(r)),
$$
so
$$
\nabla_{e_k} e_3 = T^{-1} (\partial_r T) e_k = \frac {j'(r)}{j(r)} e_k + O\bigg(\epsilon \frac {j'(r)}{j(r)}\bigg),
$$
for $1 \leq k \leq 2$, and likewise $[e_k, e_3] = \nabla_{e_k} e_3 - \nabla_{e_3} e_k = T^{-1} \partial_r T e_k$. Taking the covariant derivative
$$
\nabla_{e_k} (e_\ell \cdot e_\ell) = \nabla_{e_k} 1 = 0
$$
shows that
\be\lb{simple}
(\nabla_{e_k} e_\ell) \cdot e_\ell = 0,
\ee
so $\nabla_{e_k} e_3$ is in the plane $T_x \partial B(x_0, r)$ spanned by $e_1$ and $e_2$.

When $1 \leq k, \ell \leq 2$, taking the covariant derivative,
$$
0 = \nabla_{e_k} 0 = \nabla_{e_k} (e_\ell \cdot e_3) = (\nabla_{e_k} e_\ell) \cdot e_3 + e_\ell \cdot (\nabla_{e_k} e_3).
$$
Thus, the $e_3$ components of $\nabla_{e_k} e_\ell$ are
$$
(\nabla_{e_k} e_\ell) \cdot e_3 = -(T^{-1} \partial_r T e_k) \cdot e_\ell = -\delta_k^\ell \frac {j'(r)}{j(r)} + O(\epsilon \frac {j'(r)}{j(r)}).
$$
As in (\ref{simple}), it is always the case that $(\nabla_{e_k} e_\ell) \cdot e_\ell = 0$ and in general the components in both remaining directions can be nonzero.

By the definition of curvature, taking into account that $[e_3, J_k]=0$ and $\nabla_{e_3} e_\ell = 0$, for $1 \leq k, \ell \leq 2$
\be\lb{ref}\begin{aligned}
\partial_r (\nabla_{J_k} e_\ell) &= \nabla_{e_3} \nabla_{J_k} e_\ell = \Rm(e_3, J_k)e_\ell + \nabla_{J_k} \nabla_{e_3} e_\ell = \Rm(e_3, J_k)e_\ell.
\end{aligned}\ee
Below we obtain a more precise expression (\ref{ekl2}), but do not always need such detail.
Since by (\ref{rie}) and (\ref{rie1}) $\Rm \in L^\infty$ (as in one dimension $\dot W^{1, 1} \subset L^\infty$)
$$
\partial_r (\nabla_{J_k} e_\ell) \leq |\Rm| |J_k| \les j(r).
$$
Integrating from $x_0$ to $x$ and noting that the initial condition at $x_0$ is nonzero,
$$
\nabla_{J_k} e_\ell \les j'(r).
$$
Since the covariant derivative is linear in the first argument, by applying $T^{-1}$ to the left we obtain
\be\lb{ekl}
\nabla_{e_k} e_\ell \les \frac {j'(r)}{j(r)}.
\ee

We also need a finer estimate of $\nabla_{e_k} e_\ell$. Start from $\partial_r \nabla_{J_k} e_\ell = \Rm(e_3, J_k) e_\ell$ and
$$
\Rm(e_3, e_k) e_\ell = \delta_k^\ell \kappa_0 e_3 + O(|\Rm_1|),
$$
so $\Rm(e_3, J_k) e_\ell = \kappa_0 j_{k\ell} e_3 + O(|\Rm_1| j(r))$. The initial conditions are $\nabla_{J_k} e_\ell(x_0) = \delta_k^\ell e_3$. Then
$$
\nabla_{J_k} e_\ell = \delta_k^\ell e_3 + \kappa_0 \bigg(\int_0^r j_{k\ell} \dd s\bigg) e_3 + O(\epsilon j(r)).
$$
Consequently the projection to the tangent plane $e_1 e_2$ (orthogonal to $e_3$) fulfills
\be\lb{ekl2}
P_{e_3^\perp} \nabla_{e_k} e_\ell = O(\epsilon).
\ee


When computing $\nabla_{e_m} T$, $1 \leq m \leq 2$, we start by taking the covariant derivative (which in the case of a scalar is just the action of the vector field) in (\ref{eqt}). Since $e_3$ and $J_k$ commute, it is more appropriate to use $J_k$ than $e_k$. Recall $\mc A=(\Rm_{k33\ell})_{1 \leq k, \ell \leq 2}$ and $\mc A_0 = \kappa_0 I$. One has
\be\lb{one_derivative}
\partial^2_r (\nabla_{J_m} T) + \mc A (\nabla_{J_m} T) = -(\nabla_{J_m} \mc A) T.
\ee
Next, we use that $\nabla_{e_m} \mc A_0 = 0$, so $\nabla_{e_m} \mc A = \nabla_{e_m} (\mc A - \mc A_0)$, and compute
$$\begin{aligned}
\nabla_{e_m} \mc A_{k\ell} &= (\nabla_{e_m} \Rm)(e_k, e_3, e_3, e_\ell) + \Rm_1(\nabla_{e_m} e_k, e_3, e_3, e_\ell) + \Rm_1(e_k, \nabla_{e_m} e_3, e_3, e_\ell) + \Rm_1(e_k, e_3, \nabla_{e_m} e_3, e_\ell) + \\
&+ \Rm_1(e_k, e_3, e_3, \nabla_{e_m} e_\ell).
\end{aligned}$$
We want $\nabla_{e_m} \mc A$ to be integrable, not only small, along the geodesic $x_0x$. This will come from two sources: first, by (\ref{rie1}) the integral of $|\nabla \Rm|$ is finite along each geodesic:
$$
\int_\R |\nabla \Rm_{\gamma(s)}| \dd s < \epsilon_1.
$$
This takes care of the first term. The other terms will be small because $\Rm_1$ is small by (\ref{rie}) and we have already bounded $\nabla_{e_k} e_\ell$ by (\ref{ekl}).

Overall
\be\lb{unu}
\nabla_{e_m} \mc A \les |\nabla \Rm| + |\Rm_1| \frac {j'(r)}{j(r)},\ \nabla_{J_m} \mc A \les |\nabla \Rm| j(r) + |\Rm_1| j'(r),
\ee
and
$$
(\nabla_{J_m} \mc A) T \les |\nabla \Rm| j^2(r) + |\Rm_1| j(r)j'(r).
$$

By (\ref{expand})
$$
\nabla_{J_m} T = \frac {r^3} 6 \nabla_{J_m} \mc A(x_0, \omega) + o(r^3),\ \partial_r \nabla_{J_m} T = \frac {r^2} 2 \nabla_{J_m} \mc A(x_0, \omega) + o(r^2).
$$
As $\nabla_{J_m} \mc A(x_0, \omega)$ is of size $1$, the powers in front imply that $\partial_r \nabla_{J_m} T(r=0) = 0$ and $\nabla_{J_m} T(r=0) = 0$.


Solving (\ref{one_derivative}) with these initial conditions leads to
\be\lb{nablat}
\nabla_{J_m} T \les \epsilon j^2(r),\ \nabla_{e_m} T \les \epsilon j(r),\ \nabla_{e_m} \det T \les \epsilon j^2(r).
\ee
We also obtain
\be\lb{partialr}
\partial_r \nabla_{J_m} T \les \epsilon j^2(r).
\ee

Recall $j_{k \ell} = J_k \cdot e_\ell,\ 1 \leq k, \ell \leq 2$. Then
\be\lb{ekjl}
\nabla_{e_k} J_\ell = \sum_{m=1}^2 \nabla_{e_k} (j_{\ell m} e_m) = \sum_{m=1}^2 ((\nabla_{e_k} j_{\ell m}) e_m + j_{\ell m} \nabla_{e_k} e_m) \les j'(r),\ \nabla_{J_k} J_\ell \les j(r) j'(r).
\ee
Due to (\ref{ekl2}) the components in the $e_1e_2$ plane are better, in the sense that $P_{e_3^\perp} \nabla_{e_k} J_\ell \les \epsilon j(r)$.

We next bound $\nabla_{e_k} \nabla_{e_\ell} e_m$ and $\nabla^2_{e_k e_\ell} e_m$, for $1 \leq k, \ell \leq 2$, $1\leq m \leq 3$. Using (\ref{ekl2}), their difference is
$$
\nabla_{\nabla_{e_k} e_\ell} e_3 = \nabla_{O(\epsilon)} e_3 = O\bigg(\epsilon \frac {j'(r))}{j(r)}\bigg)
$$
and is in the $e_1e_2$ plane. Then start from
$$\begin{aligned}
\nabla_{e_k} \nabla_{e_\ell} e_3 &= \nabla_{e_k} [T^{-1} (\partial_r T) e_\ell] \\
&= -T^{-1} (\nabla_{e_k} T) T^{-1} (\partial_r T) e_\ell + T^{-1} (\partial_r \nabla_{e_k} T) e_\ell + T^{-1} (\partial_r T) \nabla_{e_k} e_\ell,
\end{aligned}$$
where, using (\ref{partialr}) and (\ref{ekl}), each term is respectively bounded by:
$$
T^{-1} (\nabla_{e_k} T) T^{-1} (\partial_r T) e_\ell \les \epsilon,\ T^{-1} (\partial_r \nabla_{e_k} T) e_\ell \les \epsilon,\ 
T^{-1} (\partial_r T) \nabla_{e_k} e_\ell \les \frac {(j'(r))^2}{j^2(r)}.
$$
This leads to
$$
\nabla^2_{e_k e_\ell} e_3,\ \nabla_{e_k} \nabla_{e_\ell} e_3 \les \frac {(j'(r))^2}{j^2(r)} \les \frac 1 {j^2(r)} + 1.
$$
In both cases, the terms in the $e_1 e_2$ plane are at most of size $\ds O\bigg(\frac {j'(r)}{j(r)}\bigg)$.

When computing $\nabla^2_{e_k e_\ell} e_m$, $1 \leq k, \ell, m \leq 2$, we start from (\ref{ref}) $\partial_r (\nabla_{J_\ell} e_m) = \Rm(e_3, J_\ell) e_m$, so
$$
\nabla_{J_k} \partial_r (\nabla_{J_\ell} e_m) = \nabla_{J_k} [\Rm(e_3, J_\ell) e_m],
$$
where
$$
\nabla_{e_k} [\Rm(e_3, J_\ell) e_m] = (\nabla_{e_k} \Rm)(e_3, J_\ell) e_m + \Rm(\nabla_{e_k} e_3, J_\ell) e_m + \Rm(e_3, \nabla_{e_k} J_\ell) e_m + \Rm(e_3, J_\ell) \nabla_{e_k} e_m
$$
and $\nabla_{J_k} = T \nabla_{e_k}$. Commuting $\nabla_{J_k}$ and $\nabla_{e_3}$ by means of the curvature tensor and since $[e_3, J_k]=0$,
\be\lb{ref2}
\partial_r \nabla_{J_k} \nabla_{J_\ell} e_m = \Rm(e_3, J_k) \nabla_{J_\ell} e_m + \nabla_{J_k} [\Rm(e_3, J_\ell) e_m]
\ee
We bound each term individually, using (\ref{ekl}) and (\ref{ekjl}):
$$
(\nabla_{e_k} \Rm)(e_3, J_\ell) e_m \les |\nabla \Rm| j(r),\ \Rm(\nabla_{e_k} e_3, J_\ell) e_m \les j'(r),\ \Rm(e_3, \nabla_{e_k} J_\ell) e_m \les j'(r),
$$
and
$$
\Rm(e_3, J_\ell) \nabla_{e_k} e_m \les j'(r).
$$
This means $\nabla_{e_k} [\Rm(e_3, J_\ell) e_m] \les j'(r)$,
hence
$$
\nabla_{J_k} [\Rm(e_3, J_\ell) e_m] \les j(r) j'(r).
$$
Finally,
$$
\Rm(e_3, J_k) \nabla_{J_\ell} e_m \les j(r) j'(r)
$$
as well. Integrating (\ref{ref2}) along the geodesic $x_0x$, note that the initial condition at $x_0$ is of size $1$: $\nabla_{J_\ell} e_m=-\delta_\ell^m e_3$ in the limit at $x_0$, then $\nabla_{J_k} e_3 = e_k$ in the limit at $r=0$. We obtain that
$$
\nabla_{J_k} \nabla_{J_\ell} e_m = -\delta_\ell^m e_k + O(j^2(r)).
$$
Then
$$
\nabla_{\nabla_{J_k} J_\ell} e_m = \nabla_{O(j(r)j'(r))} e_m \les (j'(r))^2
$$
and so
$$
\nabla^2_{J_k J_\ell} e_m = \nabla_{J_k} \nabla_{J_\ell} e_m - \nabla_{\nabla_{J_k} J_\ell} e_m \les (j'(r))^2.
$$
Applying $T^{-1}$ to the left twice, this leads to
$$
\nabla_{e_k e_\ell}^2 e_m \les \frac 1 {j^2(r)} + 1.
$$
As a consequence of (\ref{ekl2})
$$
\nabla_{\nabla_{e_k} e_\ell} e_m = \nabla_{O(\epsilon)} e_m \les \epsilon.
$$
Thus
$$
\nabla_{e_k} \nabla_{e_\ell} e_m \les \frac 1 {j^2(r)} + 1
$$
as well.

To bound $\nabla^2_{e_m} T$, $1 \leq m \leq 2$, taking two covariant derivatives in (\ref{eqt}) and recalling that $J_m$ commutes with $e_3$ we obtain
\be\lb{eq_double}
\partial^2_r (\nabla_{J_m} \nabla_{J_n} T) + \mc A (\nabla_{J_m} \nabla_{J_n} T) = -(\nabla_{J_m} \nabla_{J_n} \mc A) T - (\nabla_{J_m} \mc A) (\nabla_{J_n} T) - (\nabla_{J_n} \mc A) (\nabla_{J_m} T).
\ee

On the right-hand side, using (\ref{unu}) and (\ref{nablat})
$$
(\nabla_{J_m} \mc A) (\nabla_{J_n} T) \les \epsilon(|\nabla \Rm| j^3(r) + |\Rm_1| j^2(r) j'(r)).
$$
This leaves bounding $\nabla_{J_m} \nabla_{J_n} \mc A$. In the same manner as for the first derivative, since $\nabla_{J_m} \nabla_{J_n} \mc A_0 = 0$,
$$
\nabla_{J_m} \nabla_{J_n} \mc A = \nabla_{J_m} \nabla_{J_n} (\mc A-\mc A_0).
$$
Thus we gain a small factor of $|\Rm_1|$, integrable along the geodesic by (\ref{rie}):
$$\begin{aligned}
&\nabla^2_{e_m e_n} \mc A_{k\ell} = (\nabla^2_{e_m e_n} \Rm)(e_k, e_3, e_3, e_\ell) + \\ &+ (\nabla_{e_m} \Rm)(\nabla_{e_n} e_k, e_3, e_3, e_\ell) + (\nabla_{e_m} \Rm)(e_k, \nabla_{e_n} e_3, e_3, e_\ell) + (\nabla_{e_m} \Rm)(e_k, e_3, \nabla_{e_n} e_3, e_\ell) + (\nabla_{e_m} \Rm)(e_k, e_3, e_3, \nabla_{e_n} e_\ell) \\
&+ (\nabla_{e_n} \Rm)(\nabla_{e_m} e_k, e_3, e_3, e_\ell) + (\nabla_{e_n} \Rm)(e_k, \nabla_{e_m} e_3, e_3, e_\ell) + (\nabla_{e_n} \Rm)(e_k, e_3, \nabla_{e_m} e_3, e_\ell) + (\nabla_{e_n} \Rm)(e_k, e_3, e_3, \nabla_{e_m} e_\ell) \\
&+ \Rm_1(\nabla^2_{e_m e_n} e_k, e_3, e_3, e_\ell) + \Rm_1(e_k, \nabla^2_{e_m e_n} e_3, e_3, e_\ell) + \Rm_1(e_k, e_3, \nabla^2_{e_m e_n} e_3, e_\ell) + \Rm_1(e_k, e_3, e_3, \nabla^2_{e_m e_n} e_\ell) \\
&+ \Rm_1(\nabla_{e_m} e_k, \nabla_{e_n} e_3, e_3, e_\ell) + \Rm_1(\nabla_{e_m} e_k, e_3, \nabla_{e_n} e_3, e_\ell) + \Rm_1(\nabla_{e_m} e_k, e_3, e_3, \nabla_{e_n} e_\ell) \\
&+ \Rm_1(e_k, \nabla_{e_m} e_3, \nabla_{e_n} e_3, e_\ell) + \Rm_1(e_k, e_3, \nabla_{e_m} e_3, \nabla_{e_n} e_\ell) + \Rm_1(e_k, \nabla_{e_m} e_3, e_3, \nabla_{e_n} e_\ell) \\
&+ \Rm_1(\nabla_{e_n} e_k, \nabla_{e_m} e_3, e_3, e_\ell) + \Rm_1(\nabla_{e_n} e_k, e_3, \nabla_{e_m} e_3, e_\ell) + \Rm_1(\nabla_{e_n} e_k, e_3, e_3, \nabla_{e_m} e_\ell) \\
&+ \Rm_1(e_k, \nabla_{e_n} e_3, \nabla_{e_m} e_3, e_\ell) + \Rm_1(e_k, e_3, \nabla_{e_n} e_3, \nabla_{e_m} e_\ell) + \Rm_1(e_k, \nabla_{e_n} e_3, e_3, \nabla_{e_m} e_\ell).
\end{aligned}$$
Condition (\ref{rie2}) states that $\nabla^2 \Rm$ should be integrable along any geodesic, in particular along $x_0 x$:
$$
\int_\R |\nabla^2 \Rm_{\gamma(s)}| \dd s < \epsilon_2.
$$
Then the first nine terms can be handled as in (\ref{unu}), leading to a bound of
$$
\les |\nabla^2 \Rm| + |\nabla \Rm| \frac {j'(r)}{j(r)}.
$$
The last sixteen terms fulfill the bound
$$
\les |\Rm_1| (1+\frac 1 {j^2(r)}).
$$
We obtain that
$$
\nabla^2_{e_m e_n} \mc A \les |\nabla^2 \Rm| + |\nabla \Rm| \frac {j'(r)}{j(r)} + |\Rm_1|(1+ \frac 1 {j^2(r)}).
$$
Likewise
$$
\nabla_{\nabla_{e_m} e_n} \mc A = \nabla_{O(\frac {j'(r)}{j(r)})} \mc A \les \frac {j'(r)}{j(r)} (|\nabla \Rm| + |\Rm_1| \frac {j'(r)}{j(r)})
$$
and
$$
\nabla_{\nabla_{J_m} J_n} \mc A = \nabla_{O(j(r)j'(r))} \mc A \les j(r)j'(r) (|\nabla \Rm| + |\Rm_1| \frac {j'(r)}{j(r)}).
$$
Hence
$$
\nabla_{e_m} \nabla_{e_n} \mc A \les |\nabla^2 \Rm| + |\nabla \Rm| \frac {j'(r)}{j(r)} + |\Rm_1|(1+ \frac 1 {j^2(r)})
$$
and
$$
\nabla^2_{J_m J_n} \mc A,\ \nabla_{J_m} \nabla_{J_n} \mc A \les |\nabla^2 \Rm| j^2(r) + |\nabla \Rm| j(r)j'(r) + |\Rm_1| (j'(r))^2
$$
as well. Then
$$
(\nabla^2_{J_m J_n} \mc A) T,\ (\nabla_{J_m} \nabla_{J_n} \mc A) T \les |\nabla^2 \Rm| j^3(r) + |\nabla \Rm| j^2(r)j'(r) + |\Rm_1| j(r) (j'(r))^2.
$$

Near the origin $x_0$, using the expansion (\ref{expand}) again, we also get that
\be\lb{expansion}
\nabla_{J_m} \nabla_{J_n} T = \frac {r^3} 6 \nabla_{J_m} \nabla_{J_n} \mc A(x_0, \omega) + o(r^3),\ \partial_r \nabla_{J_m} \nabla_{J_n} T = \frac {r^2} 2 \nabla_{J_m} \nabla_{J_n} \mc A(x_0, \omega) + o(r^2).
\ee
Thus both of these initial conditions are null at $x_0$.

%

Solving equation (\ref{eq_double}) with null initial conditions and recalling that integrating against (\ref{propk}) results in a gain of $r^2$ over the forcing term for small $r$ results in
$$
\nabla_{J_m} \nabla_{J_n} T \les \epsilon j^3(r).
$$
Consequently
\be\lb{consec}
\nabla_{\nabla_{e_m} J_n} T + \nabla^2_{e_m J_n} T = \nabla_{e_m} (\nabla_{J_n} T) \les \epsilon j^2(r).
\ee
Write
$$
\nabla_{e_m} J_n = \sum_{p=1}^2 \nabla_{e_m} (j_{np} e_p) = (\sum_{p=1}^2 (\nabla_{e_m} j_{np}) e_p) + (\sum_{p=1}^2 j_{np} \nabla_{e_m} e_p) = O(\epsilon j(r)) + (\sum_{p=1}^2 j_{np} \nabla_{e_m} e_p).
$$
Then (\ref{consec}) becomes
\be\lb{precedent2}
\sum_{p=1}^2 j_{np} [(\nabla_{e_m} e_p) T + \nabla^2_{e_m e_p} T] = \sum_{p=1}^2 j_{np} (\nabla_{e_m} e_p) T + \sum_{p=1}^2 j_{np} \nabla^2_{e_m e_p} T \les \epsilon j^2(r),
\ee
hence
$$
\nabla_{e_m} \nabla_{e_p} T = (\nabla_{e_m} e_p) T + \nabla^2_{e_m e_p} T \les \epsilon j(r)
$$
and finally
\be\lb{final}
\nabla_{e_m} \nabla_{e_n} \det T \les \epsilon j^2(r).
\ee

Next, compute $\Delta f$: recalling that $a=J_1 \wedge J_2=\det T$ and $f=a^{-1/2}$,
\be\lb{deltaf}
\nabla f = -\frac {\nabla a}{2a^{3/2}},\ \Delta f = -\frac {\Delta a}{2a^{3/2}} + \frac {3|\nabla a|^2}{4a^{5/2}}.
\ee
We express $\nabla a$ in the orthonormal basis $\{e_1, e_2, e_3\}$ as
$$
\nabla a = (\nabla_{e_1} a) e_1 + (\nabla_{e_2} a) e_2 + (\nabla_{e_3} a) e_3,\ |\nabla a|^2 = (\nabla_{e_1} a)^2 + (\nabla_{e_2} a)^2 + (\nabla_{e_3} a)^2.
$$
By (\ref{both}) and (\ref{improv})
$$
\nabla_{e_3} a = \partial_r a = 2 j(r) j'(r) + O(\epsilon j^2(r)).
$$
For the other two partial derivatives, by (\ref{nablat}) $\nabla_{e_m} \det T \les \epsilon j^2(r)$, $1 \leq m \leq 2$. Thus
$$
|\nabla a|^2 = (\partial_r a)^2 + O(\epsilon^2 j^4(r)) = 4 j^2(r) (j'(r))^2 + O(\epsilon j^3(r) j'(r)).
$$

If the gradient is expressed in the basis $\{e_1, e_2, e_3\}$ as $\nabla a = \sum_{k=1}^3 (\nabla_{e_k} a) e_k$, then its divergence is
\be\lb{delta}
\Delta a = \div \nabla a = \nabla_{e_1} \nabla_{e_1} a + \nabla_{e_2} \nabla_{e_2} a + \nabla_{e_3} \nabla_{e_3} a + (\nabla_{e_1} a) \div e_1 + (\nabla_{e_2} a) \div e_2 + (\nabla_{e_3} a) \div e_3.
\ee

Equivalently, we start from the expression
\be\lb{expr}
\Delta a = \nabla^2_{e_1} a + \nabla^2_{e_2} a + \nabla^2_{e_3} a = \nabla_{e_1} \nabla_{e_1} a - \nabla_{\nabla_{e_1} e_1} a + \nabla_{e_2} \nabla_{e_2} a - \nabla_{\nabla_{e_2} e_2} a + \nabla_{e_3} \nabla_{e_3} a.
\ee
Taking the covariant derivative
$$
(\nabla_{e_1} e_1) \cdot e_2 = -(\nabla_{e_1} e_2) \cdot e_1,\ (\nabla_{e_1} e_1) \cdot e_3 = -(\nabla_{e_1} e_3) \cdot e_1,\ (\nabla_{e_2} e_2) \cdot e_1 = -(\nabla_{e_2} e_1) \cdot e_2,\ (\nabla_{e_2} e_2) \cdot e_3 = -(\nabla_{e_2} e_3) \cdot e_2
$$
so (\ref{expr}) becomes
$$
\nabla_{e_1} \nabla_{e_1} a + \nabla_{e_2} \nabla_{e_2} a + \nabla_{e_3} \nabla_{e_3} a + [(\nabla_{e_1} e_2) \cdot e_1] \nabla_{e_2} a + [(\nabla_{e_1} e_3) \cdot e_1] \nabla_{e_3} a + [(\nabla_{e_2} e_1) \cdot e_2] \nabla_{e_1} a + [(\nabla_{e_2} e_3) \cdot e_2] \nabla_{e_3} a
$$
and grouping the terms we obtain again (\ref{delta}).

Using (\ref{ekl2}) we compute
$$
\div e_1 = (\nabla_{e_1} e_1) \cdot e_1 + (\nabla_{e_2} e_1) \cdot e_2 + (\nabla_{e_3} e_1) \cdot e_3 = (\nabla_{e_2} e_1) \cdot e_2 = O(\epsilon).
$$
Note again that $(\nabla_{e_1} e_1) \cdot e_1 = 0$. Likewise $\ds\div e_2 = O(\epsilon)$ and
$$
\div e_3 = (\nabla_{e_1} e_3) \cdot e_1 + (\nabla_{e_2} e_3) \cdot e_2 + (\nabla_{e_3} e_3) \cdot e_3 = \frac {\partial_r a} a = 2\frac {j'(r)}{j(r)} + O\bigg(\epsilon \frac {j'(r)}{j(r)}\bigg).
$$
Using (\ref{nablat})
$$\begin{aligned}
(\nabla_{e_1} a) \div e_1 + (\nabla_{e_2} a) \div e_2 + (\nabla_{e_3} a) \div e_3 &= O(\epsilon j^2(r)) O(\epsilon) + \frac {(\partial_r a)^2} a = \frac {(\partial_r a)^2} a + O(\epsilon^2 j^2(r)).
\end{aligned}$$

One has
$$\begin{aligned}
\nabla_{e_3} \nabla_{e_3} a = \nabla_{e_3}^2 a = \partial^2_r a &= \partial^2_r J_1 \wedge J_2 + 2 \det \partial_r T + J_1 \wedge \partial^2_r J_2 \\
&= -(\mc A J_1) \wedge J_2 + 2 \det \partial_r T - J_1 \wedge (\mc A J_2) \\
&= -2\kappa_0 a + 2 \det \partial_r T + O(|\Rm_1| j^2(r)),
\end{aligned}$$
so by (\ref{final})
$$
\nabla_{e_1} \nabla_{e_1} a + \nabla_{e_2} \nabla_{e_2} a + \nabla_{e_3} \nabla_{e_3} a = -2\kappa_0 a + 2 \det \partial_r T + O(|\Rm_1| j^2(r)) + O(\epsilon j(r) j'(r)).
$$
Again, $|\Rm_1| \leq \|\nabla \Rm\|_{L^1(\gamma)} < \epsilon$, using the Sobolev embedding $\dot W^{1, 1} \subset L^\infty$ in one dimension. We end up with
$$\begin{aligned}
\Delta a &= -2\kappa_0 a + 2 \det \partial_r T + O(|\Rm_1| j^2(r)) + O(\epsilon j^2(r)) + \frac {(\partial_r a)^2} a + O(\epsilon^2 j^2(r)) \\
&= -2\kappa_0 a + 2 \det \partial_r T + \frac {(\partial_r a)^2} a + O(\epsilon j^2(r)).
\end{aligned}$$

The following estimate for $\partial_r J_1 \wedge \partial_r J_2 = \det \partial_r T$ has an improvement for $r<1$ due to (\ref{smallr}):
$$
\partial_r J_1 \wedge \partial_r J_2 = \det \partial_r T = (j'(r))^2 + O(\epsilon j^2(r)).
$$
Hence, using both (\ref{both}) and (\ref{improv}) for $r<1$, plugging our computations into (\ref{deltaf})
$$\begin{aligned}
\Delta f &= \frac {2\kappa a - 2 \det \partial_r T - \frac {(\partial_r a)^2} a + O(\epsilon j^2(r))}{2a^{3/2}} + \frac {3(\partial_r a)^2 + O(\epsilon^2 j^4(r))} {4a^{5/2}} \\
&= \frac {\kappa_0} {a^{1/2}} + \frac {(\partial_r a)^2-4a\det \partial_r T}{4a^{5/2}} + O\Big(\frac {\epsilon} {j(r)}\Big) = \kappa_0 f + O\Big(\frac {\epsilon} {j(r)}\Big).
\end{aligned}$$

Denote the error term above by $\ds\Error(x, x_0) = O\Big(\frac \epsilon {j(r)} \Big)$. Then
$$
\Delta f = \kappa_0 f + \Error(x, x_0)
$$
and $u(x, t)=f(x) \chi(r-t)$ given by (\ref{fx}) solves the wave equation (\ref{freemod}) with an error of
$$
\partial^2_t u - (\Delta-\kappa_0) u = \chi(r-t) \Error(x, x_0).
$$
In the same way as in (\ref{approx}), by taking a limit in $\chi$, this implies that for $t \geq 0$
$$
S_0(t)(x, x_0) = \frac 1 {4\pi} f(x) \delta_{d(x_0, x)=t} = \frac 1 {4\pi \sqrt {a(x, x_0)}} \delta_{d(x_0, x) = t}
$$
approximately solves the free modified wave equation (\ref{freemod}):
\be\lb{approx2}
\partial_t^2 S_0 - (\Delta-\kappa_0) S_0 = E,\ S_0(0)=0,\ \partial_t S_0(0) = \delta_{x_0},
\ee
where
$$
E(t)(x, x_0) = \Error(x, x_0) \delta_{d(x_0, x)=t}.
$$
Both the approximate solution $S_0(t)$ and the inhomogeneous error term $E(t)$ are measures supported on the light cone, meaning on $\partial B(x_0, t)$ for each $t>0$. The error term has size $\ds O\Big(\frac {\epsilon} {j(t)}\Big)$:
$$
\|\Error(x, x_0)\|_{L^\infty(\partial B(x_0, t))} \les \frac \epsilon {j(t)},
$$
hence reduces to $0$ in the constant sectional curvature case treated in the previous section.

By iterating we construct a sequence of increasingly better approximations to the sine propagator. First, we use the approximate solution $S_0$ to propagate the error, resulting in the second iteration
$$
S_1=S_0+S_0 * E,
$$
where $S_0 \ast E$ is defined, see (\ref{s0e1}), as
\be\lb{s0e}\begin{aligned}
[S_0 * E](t)(x_1, x_0) &= \int_0^t \int_{\mc M} S_0(t-s)(x_1, x) E(s)(x, x_0) \dd x \dd s \\
&= \int_0^t \int_{\mc M} \frac 1 {4\pi\sqrt {a(x_1, x)}} \delta_{d(x, x_1)=t-s} \Error(x, x_0) \delta_{d(x_0, x)=s} \dd x \dd s.
\end{aligned}\ee
Then $S_1$ is also an approximate solution to the wave equation, with error $E * E$:
$$
\partial^2_t S_1-(\Delta-\kappa_0)S_1=E * E,\ S_1(0)=0,\ \partial_t S_1(0) = \delta_{x_0}.
$$
We again apply the approximate propagator $S_0$ to $E*E$ and keep iterating. Further iterations take the form
$$
S_n=S_0+S_0 * E + S_0 * E * E + \ldots + S_0 * \underbrace {E * E * \ldots * E}_{n \text{ times}}
$$
and approximately satisfy equation (\ref{freemod}) with error
\be\lb{multerror}
\underbrace{E * E * \ldots * E}_{n+1 \text{ times}}.
\ee
The $L^1$ norm of $S_0(t)$ is of size $O(j(t))$, while the $L^1_x$ norm of $\Error(t)$ is of size $O(\epsilon j(t))$. Then, for fixed $T>0$, the $L^\infty([0, T]) L^1_x$ norm of the $n$-th error term (\ref{multerror}) is of size
$$
C^{n+1} \epsilon^{n+1} \int_{0 \leq s_1 \leq s_2 \leq \ldots \les s_n \leq t} j(s_1) j(s_2-s_1) \ldots j(t-s_n) \dd s_1 \ldots \dd s_n \les \frac {C^{n+1} \epsilon^{n+1}}{2^n} \frac {T^n}{n!} j(T) \to 0
$$
as $n \to \infty$. At the same time, $S_n$ form a Cauchy sequence in the same norm, with each term bounded in $L^\infty([0, T]) L^1_x$ by
$$
\frac {C^n \epsilon^n}{2^n} \frac {T^n}{n!} j(T)
$$
and an overall bound of
\be\lb{bddd}
\bigg(\sum_{n=0}^\infty \frac {C^n \epsilon^n}{2^n} \frac {T^n}{n!} \bigg) j(T) \leq e^{C\epsilon T/2} j(T).
\ee
It follows that the approximating sequence $S_n$ converges to a solution $S$ of (\ref{freemod}), the sum of the series
\be\lb{sseries}
S=S_0+S_0 * E + S_0 * E * E + \ldots
\ee
$S$ is the sine propagator of equation (\ref{freemod}), whose $L^1$ norm is also bounded by (\ref{bddd}).

%
%
%

In conclusion, for $t>0$ the propagator $S(t)(x_0, x)$ is given by a measure supported on $B(x_0, t)$, of total mass at most $O(e^{c\epsilon t} j(t))$. This measure is absolutely continuous except for the component $S_0$ supported on $\partial B(x_0, t)$.
\end{proof}

\subsection{Lebesgue spaces on the manifold}
By (\ref{both}) we have ensured that the volume and surface area of a metric ball are comparable to those for constant sectional curvature $\kappa_0$, up to factors of $1+O(\epsilon)$:
$$
\int_{B(x_0, r)} F \sim \int_{\S^2} \int_0^r F(s, \omega) j^2(s) \dd s \dd \omega
$$
and
$$
\int_{\partial B(x_0, r)} F \sim \int_{\S^2} F(r, \omega) j^2(r) \dd \omega.
$$
Thus, instead of looking along geodesics, we can use Lebesgue spaces to measure the size of the metricq perturbation $\Rm_1$.

One can characterize Lebesgue spaces $L^p$, $1 \leq p \leq \infty$, and Lorentz spaces in spherical coordinates. Unlike in the flat metric case $\R^3$, where they are tied to polynomial decay, here Lebesgue spaces are related to exponential decay. For example, $\ds \frac {G(r)}{j^2(r)} \in L^1(\mc M)$ if and only if $G \in L^1([0, \infty))$.

Likewise, $\ds \frac 1 {j^2(r)} \in L^{1, \infty}$. Indeed, for any $\alpha>0$, $\{x \mid |G(x)|>\alpha\} = B(x_0, r_0)$, where $r_0$ is such that $\ds j^2(r_0)=\frac 1 \alpha$. But then the volume of the ball is at most
$$
\les \int_0^{r_0} j^2(r) \dd r \les \int_0^{r_0} j(r)j'(r) \dd r \les j^2(r_0) = \frac 1 \alpha.
$$
Consequently $\ds \frac 1 {j(r)} \in L^{2, \infty}$.

Following \cite{RodSch}, we introduce the global Kato space (\ref{kato})
$$
\K = \{f \in L^1_{loc} \mid \sup_{x_0 \in \mc M} \int_{\mc M} \frac {|V(x)| \dd x}{j(d(x, x_0))} < \infty \},
$$
with the corresponding norm. Clearly, $L^{2-\epsilon} \cap L^{2+\epsilon} \subset L^{2, 1} \subset \K$ (compare to the flat metric case $\R^3$, where $L^{3/2-\epsilon} \cap L^{3/2+\epsilon} \subset L^{3/2, 1} \subset \K$), but $\K$ does not require local square integrability. Also let $\tilde \K = \K \cap L^1$.

In addition to $\K$ and $\tilde \K$, consider the modified Kato space $\K_\delta$, for small $\delta>0$, obtained by using
\be\lb{jdelta}
j_\delta(r)=\frac {\sinh(r(\alpha_0-\delta))}{\alpha_0-\delta}
\ee
instead of $j$ (\ref{jneg}) in the definition of the Kato space (\ref{kato}). Since $\tilde \K \subset \K_\delta \subset \K$, it follows that $\K^* \subset \K_\delta^*$.

As in \cite{BecGol1} and \cite{BecGol2}, for any two Banach lattices $X$ and $Y$ introduce the space
\be\lb{u}
\U(X, Y) = \{T(t, y, x) \mid \int_\R |T(t, y, x)| \dd \tau \in \B(X, Y)\}.
\ee
This structure is an algebroid under the operation $*$, see (\ref{s0e}),
\be\lb{s0e1}\begin{aligned}
[F * G](t)(x_1, x_0) &= \int_\R \int_{\mc M} F(t-s)(x_1, x) G(s)(x, x_0) \dd x \dd s.
\end{aligned}\ee
This operation is not commutative, but it is associative, and in general
$$
\|S * T\|_{\U(X, Z)} \leq \|S\|_{\U(Y, Z)} \|T\|_{\U(X, Y)}.
$$
Also let $\U(X) = \U(X, X)$, which is an algebra.

\begin{proposition}\lb{prop3} Consider a simply connected three-dimensional Riemannian manifold $(\mc M, g)$ such that the Riemann curvature tensor $\Rm$ satisfies conditions (\ref{rie})
$$
\|\Rm_1\|_{L^1(\Gamma)} = \sup_{\gamma \in \Gamma} \int_\R |\Rm_{\gamma(s)}-\Rm_0| \dd s < \epsilon << 1,
$$
(\ref{rie1}), and (\ref{rie2}) and such that
$$
\|\Rm_1\|_{L^1} < \epsilon,\ \|\nabla \Rm\|_{L^1} < \epsilon,\ \|\nabla^2 \Rm\|_{L^1} < \epsilon.
$$
Then the sine propagator $S(t)(x, x_0)$ for the free modified wave equation (\ref{freemod}) is in $\U(L^1, \K^*) \cap \U(\K, L^\infty)$ and $tS(t) \in \U(L^1, L^\infty)$.

Furthermore, for any $\delta>0$, if $\epsilon>0$ is sufficiently small,
$$
j_\delta(t) S(t) \in \U(L^1, L^\infty),\ j'_\delta(t) S(t) \in \U(L^1, \tilde \K^*) \cap \U(\tilde \K, L^\infty).
$$
\end{proposition}

\begin{proof}
For $f=a^{-1/2}$, recall
\be\lb{delta2}
\Delta f = -\frac {\Delta a}{2a^{3/2}} + \frac {3|\nabla a|^2}{4a^{5/2}},
\ee
where we consider each term individually in
$$
\Delta a = \nabla_{e_1} \nabla_{e_1} a + \nabla_{e_2} \nabla_{e_2} a + \nabla_{e_3} \nabla_{e_3} a + (\div e_1) \nabla_{e_1} a + (\div e_2) \nabla_{e_2} a + (\div e_3) \nabla_{e_3} a.
$$
As in the proof of the previous proposition, the goal is writing $\Delta f = \kappa_0 f + \Error$, but this time with improved integrability estimates for $\Error$.

We split the manifold $\mc M$ into two regions, as a function of $r=d(x_0, x)$. For $r<1$, i.e.~on $B(x_0, 1)$, we use the estimates of Proposition \ref{prop2}, leading to
$$
\Delta f = \kappa_0 f + O\bigg(\frac \epsilon r\bigg),
$$
where the error term is integrable on $B(x_0, 1)$, with a bound independent of $x_0$.

For $r>1$ we take advantage of the fact that $\Rm_1 = \Rm - \Rm_0 \in L^1$ is small in norm
\be\lb{l1norm}
\|\Rm_1\|_{L^1} < \epsilon << 1
\ee
and for future estimates also use $\|\nabla \Rm\|_{L^1} < \epsilon$ and $\|\nabla^2 \Rm\|_{L^1} < \epsilon$, together with similar bounds in the Kato space $\K$. This leads to a different set of bounds for the error terms in (\ref{delta2}), not over each individual geodesic as in Proposition \ref{prop2}, but averaged over all geodesics starting at $x_0$.

In spherical coordinates, (\ref{l1norm}) and (\ref{both}) mean
$$
\int_{\S^2} \int_0^\infty |\Rm(r, \omega)-\Rm_0| j^2(r) \dd r \dd \omega < \epsilon
$$
and same for $\nabla \Rm$ and $\nabla^2 \Rm$.
This means exponential decay, in average, over geodesics starting from $x_0$.


Under these conditions, a more precise estimate for $J_k$ comes from
\be\lb{aici}
\frac {J_k(r, \omega)} {j(r)} = e_k - \int_0^r \frac {j(r-s)}{j(r)} [\mc A_1 J_k](s) \dd s.
\ee
Here $\ds\frac {j(r-s)}{j(r)} \leq \frac 1 {j'(s)}$ and
$$
\int_0^\infty \frac 1 {j'(s)} [\mc A_1 J_k](s) \dd s \les \int_0^\infty |\Rm_1| \dd s < \infty.
$$
We use a scattering approach here. For $1 \leq k \leq 2$, let
\be\lb{ik}
I_k(\omega) = \int_0^\infty e^{-s\alpha_0} [\mc A_1 J_k](s, \omega) \dd s,
\ee
where the integral is taken along the geodesic starting at $x_0$ in the direction $\omega \in \S^2 \subset T_{x_0} \mc M$, that is on $\{\exp_{x_0}(s\omega) \mid s \in [0, \infty)\}$.

Taking the covariant derivative $\nabla_{e_k} I_\ell$, $1 \leq k, \ell \leq 2$, leads to taking the covariant derivative with respect to $J_k$ on the right-hand side of (\ref{ik}):
$$
\nabla_{e_k} I_\ell(\omega) = \int_0^\infty e^{-s\alpha_0} [(\nabla_{J_k} \mc A) J_\ell + \mc A_1 \nabla_{J_k} J_\ell](s, \omega) \dd s.
$$
For future reference
$$
|I_k| \les \|\Rm_1\|_{L^1(\Gamma)},\ \int_{\S^2} |\nabla_{e_k} I_\ell| \les \|\Rm_1\|_\K + \|\nabla \Rm\|_\K,
$$
and
\be\lb{second_order}
\int_{\S^2} |\nabla_{e_k} \nabla_{e_\ell} I_m| \les \|\Rm_1\|_{L^1} + \|\nabla \Rm\|_{L^1} + \|\nabla^2 \Rm\|_{L^1}.
\ee
Here $\omega \in \S^2 \subset T_{x_0} \mc M$ and $e_k \in T_\omega (T_{x_0} \mc M) = T_{x_0} \mc M$. However, $e_k$ do not commute with $e_3$, so the corresponding derivatives at $x \in \mc M$, where $d(x_0, x) = r$, have to be replaced by $J_k$.

We extend $I_k(\omega)$, initially defined on $\S^2 \subset T_{x_0} \mc M$, to a function on $\mc M$, constant along geodesics starting from $x_0$, which by abuse of notation we call $I_k(r, \omega)=I_k(\omega)$. Then
$$
\nabla_{J_k} I_\ell(r, \omega) = \nabla_{e_k} I_\ell(\omega).
$$
We do the same for $T_\infty$ (\ref{tinf}).

A useful $L^2$ estimate for $\nabla_{e_k} I_\ell$ follows from (\ref{second_order}) by the two-dimensional $\S^2$ Sobolev embedding $W^{1, 1} \subset L^2$:
\be\lb{sobolev}
\|f - \ov f\|_{L^2(\S^2)} \les \|\nabla f\|_{L^1(\S^2)},
\ee
where $\ov f$ is the average of $f$ on $\S^2$, with a constant that only depends on $\S^2$. Then
\be\lb{l2}
\|\nabla_{e_k} I_\ell\|_{L^2(\S^2)} \les \|\Rm_1\|_{L^1 \cap \K} + \|\nabla \Rm\|_{L^1 \cap \K} + \|\nabla^2 \Rm\|_{L^1}.
\ee

For $r \geq 1$,
\be\lb{aux}
\frac {j(r-s)}{j(r)} - e^{-s\alpha_0} = \frac {e^{\alpha_0(r-s)} - e^{\alpha_0(s-r)}}{e^{r\alpha_0}-e^{-r\alpha_0}} - e^{-s\alpha_0} \les e^{(s-2r)\alpha_0},
\ee
so
$$\begin{aligned}
\bigg|\int_0^r \frac {j(r-s)}{j(r)} [\mc A_1 J_k](s) \dd s - I_k\bigg| &\les \int_r^\infty e^{-s\alpha_0} [\mc A_1 J_k](s) \dd s + \int_0^r e^{(s-2r)\alpha_0} [\mc A_1 J_k](s) \dd s \\
&\les e^{-2r\alpha_0} \int_0^\infty |\Rm_1| j^2(s) \dd s.
\end{aligned}$$
Averaging over all directions $\omega \in \S^2 \subset T_{x_0} \mc M$, the right-hand side is controlled by $e^{-2r} \|\Rm_1\|_{L^1}$. Hence
\be\lb{l11}
\int_{\S^2} \bigg|\frac {J_k(r, \omega)}{j(r)} - e_k + I_k(\omega)\bigg| \les e^{-2r \alpha_0} \|\Rm_1\|_{L^1}.
\ee

Likewise, a more precise estimate for $\partial_r J_k$ comes from
$$
\frac {\partial_r J_k(r, \omega)-j'(r)e_k} {j(r)} = e_k - \int_0^r \frac {j'(r-s)}{j(r)} [\Rm_1(J_k, e_3)e_3](s) \dd s.
$$
Note that $\ds\frac {j'(r-s)}{j(r)} \leq \frac 1 {j(s)}$ and
$$
\int_0^\infty \frac 1 {j(s)} [\Rm_1(J_k, e_3)e_3](s) \dd s \les \int_0^\infty |\Rm_1| \dd s < \infty,
$$
while
$$
\frac {j'(r-s)}{j(r)} - \alpha_0 e^{-s\alpha_0} \les e^{(s-2r)\alpha_0}.
$$
Consequently
$$\begin{aligned}
	&\bigg|\int_0^r \frac {j'(r-s)}{j(r)} [\Rm_1(J_k, e_3)e_3](s) \dd s - \alpha_0 I_k \bigg| \les \\
	&\les \int_r^\infty e^{-s\alpha_0} [\Rm_1(J_k, e_3)e_3](s) \dd s + \int_0^r e^{(s-2r)\alpha_0} [\Rm_1(J_k, e_3)e_3](s) \dd s \les e^{-2r\alpha_0} \int_0^\infty |\Rm_1| j^2(s) \dd s,
\end{aligned}$$
leading to
\be\lb{l12}
\int_{\S^2} \bigg|\frac {\partial_r J_k(r) - j'(r) e_k}{j(r)} + \alpha_0 I_k(\omega)\bigg| \dd \omega \les e^{-2r\alpha_0} \|\Rm_1\|_{L^1}.
\ee

In terms of the matrix $T$ (\ref{t}), we construct
\be\lb{tinf}
T_\infty(r, \omega) = T_\infty(\omega)=\begin{pmatrix}
1-I_{11}(\omega) & -I_{21}(\omega) \\
-I_{12}(\omega) & 1-I_{22}(\omega)
\end{pmatrix},
\ee
where
$$
I_k(\omega) = \sum_{\ell=1}^2 I_{k\ell}(\omega) e_\ell,\ I_{k\ell}(\omega) = I_k(\omega) \cdot e_\ell,\ 1 \leq k, \ell \leq 2.
$$
Then $\det T_\infty = (e_1-I_1) \wedge (e_2-I_2)$ and
$$
\int_{\S^2} \bigg|\frac T {j(r)} - T_\infty\bigg| \les e^{-2r\alpha_0} \|\Rm_1\|_{L^1},\ \int_{\S^2} \bigg|\frac {\partial_r T-j'(r) T_\infty}{j(r)}\bigg| \les e^{-2r\alpha_0} (\|\Rm_1\|_{L^1} + \|\Rm_1\|_{L^1(\Gamma)}).
$$
In the second inequality, the terms corresponding to $I_k$ are $\alpha_0 j(r) I_k$ and the difference with $j'(r) I_k$ is of size
$$
e^{-2r\alpha_0} I_k \les e^{-2r\alpha_0} \|\Rm_1\|_{L^1(\Gamma)},
$$
which appears on the right-hand side.


Finally,
\be\lb{final1}
\int_{\S^2} |a(r, \omega)-j^2(r) \det T_\infty| \dd \omega \les \|\Rm_1\|_{L^1},\ \int_{\S^2} |\partial_r a(r, \omega) - 2 j(r) j'(r) \det T_\infty| \dd \omega \les \|\Rm_1\|_{L^1},
\ee
and
\be\lb{final2}
\int_{\S^2} |\det \partial_r T(r, \omega)-(j'(r))^2 \det T_\infty| \dd \omega \les \|\Rm_1\|_{L^1}.
\ee

Using $2j(r)j'(r) \det T_\infty$ instead of $2j(r)j'(r)$ leads to a large improvement in the error bounds. Thus we can capture most of the error in (\ref{both}) in a precise way, in average.

Assuming that $\|\nabla \Rm\|_{L^1}<\infty$ and $\|\Rm_1\|_{L^1}<\infty$, for a better estimate of $\nabla_{e_k} a$ write as in (\ref{aici})
$$
J_k(r, \omega) = j(r) e_k - \int_0^r j(r-s) [\mc A_1 J_k](s) \dd s
$$
and
$$
\frac {J_k(r, \omega)}{j(r)} - (e_k - I_k(\omega)) = \int_r^\infty e^{-s\alpha_0} [\mc A_1 J_k](s) - \int_0^r (\frac {j(r-s)}{j(r)} - e^{-s\alpha_0}) [\mc A_1 J_k](s) \dd s.
$$
In terms of $T$, this means
$$
T(r, \omega) = j(r) I - \int_0^r j(r-s) [\mc A_1 T](s) \dd s
$$
and
\be\lb{tmp1}
\frac {T(r, \omega)}{j(r)} - T_\infty(\omega) = \int_r^\infty e^{-s\alpha_0} [\mc A_1 T](s) \dd s - \int_0^r (\frac {j(r-s)}{j(r)} - e^{-s\alpha_0}) [\mc A_1 T](s) \dd s.
\ee


Applying the vector fields $J_m$, which commute with $\partial_r$, to (\ref{tmp1}) produces
$$\begin{aligned}
\frac {\nabla_{J_m} T(r, \omega)}{j(r)} - \nabla_{J_m} T_\infty(r, \omega) &= \int_r^\infty e^{-s\alpha_0} [(\nabla_{J_m} \mc A) T + \mc A_1 (\nabla_{J_m} T)](s) \dd s \\
&- \int_0^r (\frac {j(r-s)}{j(r)} - e^{-s\alpha_0}) [(\nabla_{J_m} \mc A) T + \mc A_1 (\nabla_{J_m} T)](s) \dd s.
\end{aligned}$$
Using (\ref{aux}), we bound the right-hand side by
$$
\les \int_r^\infty e^{-s\alpha_0} j^2(s) (|\nabla \Rm|+|\Rm_1|) \dd s + \int_0^r e^{(s-2r)\alpha_0} j^2(s) (|\nabla \Rm|+|\Rm_1|) \dd s.
$$
Then
\be\lb{important2}
\int_{\S^2} \bigg|\frac {\nabla_{J_m} T(r, \omega)}{j(r)} - \nabla_{J_m} T_\infty(r, \omega)\bigg| \dd \omega \les \min(e^{-r\alpha_0} (\|\nabla \Rm\|_{L^1} + \|\Rm_1\|_{L^1}), \|\nabla \Rm\|_\K + \|\Rm_1\|_\K).
\ee
Together with (\ref{important}) and using the two-dimensional $\S^2$ Sobolev embedding $W^{1, 1} \subset L^2$ (\ref{sobolev}) this implies
$$
\bigg\|\frac {\nabla_{J_m} T(r, \omega)}{j(r)} - \nabla_{J_m} T_\infty(r, \omega)\bigg\|_{L^2(\S^2)} \les \|\nabla^2 \Rm\|_{L^1} + \|\nabla \Rm\|_{L^1 \cap \K} + \|\Rm_1\|_{L^1 \cap \K}
$$
and
$$
\|\nabla_{e_m} T(r, \omega) - j(r) \nabla_{e_m} T_\infty(r, \omega)\|_{L^2(\S^2)} \les \|\nabla^2 \Rm\|_{L^1} + \|\nabla \Rm\|_{L^1 \cap \K} + \|\Rm_1\|_{L^1 \cap \K}.
$$
Consequently $\nabla_{e_m} T$ is roughly the same, for large $r$, as
$$
\nabla_{e_m} T \sim j(r) \nabla_{e_m} T_\infty(r, \omega) \les \nabla_{J_m} T_\infty(r, \omega) = \nabla_{e_m} T_\infty(\omega),
$$
modulo an error term that satisfies the same bound:
$$\begin{aligned}
\int_{\S^2} |\nabla_{e_m} T(r, \omega)|^2 \dd \omega &\les \int_{\S^2} |\nabla_{e_m} T_\infty(\omega)|^2 \dd \omega + \int_{\S^2} |\nabla_{e_m} T(r, \omega) - j(r) \nabla_{e_m} T_\infty(r, \omega)|^2 \dd \omega \\
&\les \|\nabla^2 \Rm\|_{L^1} + \|\nabla \Rm\|_{L^1 \cap \K} + \|\Rm_1\|_{L^1 \cap \K}
\end{aligned}$$
and
$$
\int_{\partial B(x_0, r)} |\nabla_{e_m} T(r, \omega)|^2 \dd A \les j^2(r) \int_{\S^2} |\nabla_{e_m} T(r, \omega)|^2 \dd \omega \les j^2(r) (\|\nabla^2 \Rm\|_{L^1} + \|\nabla \Rm\|_{L^1 \cap \K} + \|\Rm_1\|_{L^1 \cap \K}),
$$
so
$$\begin{aligned}
\int_{\partial B(x_0, r)} |\nabla_{e_m} \det T(r, \omega)|^2 \dd A \les j^4(r) (\|\nabla^2 \Rm\|_{L^1} + \|\nabla \Rm\|_{L^1 \cap \K} + \|\Rm_1\|_{L^1 \cap \K}).
\end{aligned}$$

Thus the contribution of this term to the error in (\ref{delta2}), integrated over $\partial B(x_0, r)$, is of size $\ds \frac 1 {j(r)}$, hence integrable in $r>1$.


%

Taking one more covariant derivative in (\ref{tmp1})
$$\begin{aligned}
&\frac {\nabla_{J_m} \nabla_{J_n} T(r, \omega)}{j(r)} - \nabla_{J_m} \nabla_{J_n} T_\infty(r, \omega) = \\
&=\int_r^\infty e^{-s\alpha_0} [(\nabla^2_{J_m J_n} \mc A) T + \mc A_1 (\nabla^2_{J_m J_n} T) + (\nabla_{J_m} \mc A) (\nabla_{J_n} T) + (\nabla_{J_n} \mc A) (\nabla_{J_n} T)](s) \dd s \\
&- \int_0^r (\frac {j(r-s)}{j(r)} - e^{-s\alpha_0}) [(\nabla^2_{J_m J_n} \mc A) T + \mc A_1 (\nabla^2_{J_m J_n} T) + (\nabla_{J_m} \mc A) (\nabla_{J_n} T) + (\nabla_{J_n} \mc A) (\nabla_{J_n} T)](s) \dd s,
\end{aligned}$$
which using (\ref{aux}) leads to
\be\lb{important}
\int_{\S^2} \bigg|\frac {\nabla_{J_m} \nabla_{J_n} T(r, \omega)}{j(r)} - \nabla_{J_m} \nabla_{J_n} T_\infty(\omega)\bigg| \les \|\nabla^2 \Rm\|_{L^1} + \|\nabla \Rm\|_{L^1} + \|\Rm_1\|_{L^1}.
\ee
By the same procedure as in (\ref{precedent2}) and using (\ref{important2}) one obtains
$$
\int_{\S^2} \bigg|\frac {\nabla_{e_m} \nabla_{e_n} T(r, \omega)}{j(r)} - \nabla_{e_m} \nabla_{e_n} T_\infty(\omega)\bigg| \les \frac 1 {j^2(r)} (\|\nabla^2 \Rm\|_{L^1} + \|\nabla \Rm\|_{L^1} + \|\Rm_1\|_{L^1}),
$$
so
$$
\int_{\S^2} |\nabla_{e_m} \nabla_{e_n} T(r, \omega)| \dd \omega \les \frac 1 {j(r)} (\|\nabla^2 \Rm\|_{L^1} + \|\nabla \Rm\|_{L^1} + \|\Rm_1\|_{L^1}).
$$
Consequently
$$
\int_{\S^2} |\nabla_{e_m} \nabla_{e_n} \det T(r, \omega)| \dd \omega \les \|\nabla^2 \Rm\|_{L^1} + \|\nabla \Rm\|_{L^1} + \|\Rm_1\|_{L^1}
$$
and
$$
\int_{\partial B(x_0, r)} |\nabla_{e_m} \nabla_{e_n} \det T(r, \omega)| \dd A \les j^2(r) (\|\nabla^2 \Rm\|_{L^1} + \|\nabla \Rm\|_{L^1} + \|\Rm_1\|_{L^1}).
$$
Then the contribution of this term to the error in (\ref{delta2}), integrated over $\partial B(x_0, r)$, is at most of size $\ds\frac 1 {j(r)}$, which is integrable in $r>1$.

Furthermore recall
\be\lb{complicated}\begin{aligned}
\nabla_{e_3} \nabla_{e_3} a = \nabla_{e_3}^2 a = \partial^2_r a &= (\partial^2_r J_1) \wedge J_2 + 2 (\partial_r J_1) \wedge (\partial_r J_2) + J_1 \wedge (\partial^2_r J_2) =  \\
&= -(\mc A J_1) \wedge J_2 + 2 \det \partial_r T - J_1 \wedge (\mc A J_2) \\
&= -2\kappa_0 a + 2 \det \partial_r T + O(|\Rm_1|j^2(r)).
\end{aligned}\ee
The contribution of the last term in (\ref{complicated}) to the total error in (\ref{delta2}) is integrable:
$$
\int_{\mc M} \frac {|\Rm_1| j^2(r)} {a^{3/2}} \les \int_{\mc M} \frac {|\Rm_1|} {j(r)} \leq \|\Rm_1\|_\K.
$$

Next
$$
\div e_1=(\nabla_{e_2} e_1) \cdot e_2,\ \div e_2=(\nabla_{e_1} e_2) \cdot e_1
$$
are of size $O(\epsilon)$ by (\ref{ekl2}) and using (\ref{important2}) the corresponding error terms for $(\div e_1) \nabla_{e_1} a + (\div e_2) \nabla_{e_2} a$ in (\ref{delta2}) are integrable. So far we have
$$
\Delta f = \frac {\kappa_0} {a^{1/2}} + \frac {(\partial_r a)^2-4a\det \partial_r T}{4a^{5/2}} + \Error,
$$
where $\Error$ is an integrable error term, with $\|\Error\|_{L^1(\mc M)} \les \epsilon$.

For the last remaining term and $r \geq 1$, due to (\ref{final1}) and (\ref{final2})
$$
\int_{\S^2} |(\partial_r a)^2-4j^2(r)(j'(r))^2(\det T_\infty)^2| \dd \omega \les j^2(r) \|\Rm_1\|_{L^1}
$$
and
$$
\int_{\S^2} |a\det \partial_r T-j^2(r)(j'(r))^2(\det T_\infty)^2| \dd \omega \les j^2(r) \|\Rm_1\|_{L^1}
$$
hence
$$
\int_{\S^2} |(\partial_r a)^2-4a\det \partial_r T| \dd \omega \les j^2(r) \|\Rm_1\|_{L^1},
$$
so this term's contribution to the error in (\ref{delta2}), integrated over $\partial B(x_0, r)$, is again of size $\ds \frac 1 {j(r)}$, hence integrable in $r>1$.

Consequently the error term in (\ref{delta2}) is integrable:
$$
\Delta f = \kappa_0 f + \Error,
$$
where the error term $\Error$ has the following structure: in general it is of size $\ds O\Big(\frac {\epsilon} {j(r)}\Big)$, hence for $r<1$ it is of size $\ds O\Big(\frac {\epsilon} r\Big)$, by Proposition \ref{prop2}, while for $r \geq 1$
\be\lb{explicit}
\Error(x, x_0) \les \frac {\tilde E(x_0, \omega)} {j^3(r)} + \frac {|{\Rm_1}_x|}{j(r)},
\ee
where $\tilde E(x_0, \omega) \in L^1_\omega(\S^2)$, $\|\tilde E\|_{L^1_\omega(\S^2)} \les \epsilon$. The second right-hand side term comes from the last term in (\ref{complicated}).

Hence $\Error(x, x_0) \in L^1_x(\mc M)$ and $\sup_{x_0 \in \mc M} \|\Error(x, x_0)\|_{L^1_x} \les \epsilon$. Furthermore, due to the $\ds O\Big(\frac {\epsilon} {j(r)}\Big)$ bound, uniformly for $x_1 \in \mc M$
$$
\sup_{x_1 \in \mc M} \int_{B(x_1, 1)} \frac {\Error(x, x_0)}{j(d(x_1, x))} \dd x \les \epsilon,
$$
while the integral on the complement ${}^c B(x_1, 1)$ is uniformly bounded by $\|\Error\|_{L^1}$. Hence $\Error(x, x_0) \in \K_x(\mc M)$ as well and $\sup_{x_0 \in \mc M} \|\Error(x, x_0)\|_{\K_x} \les \epsilon$. In other words,
\be\lb{fir}
\Error \in \B(L^1) \cap \B(L^1, \K),
\ee
which also implies that $\Error \in \B(L^1, \tilde \K) \subset \B(\tilde \K)$, with norms of size $O(\epsilon)$.

In the same manner, due to the exponential decay, $r \Error(x, x_0) \in \B(L^1) \cap \B(L^1, \K)$ as well. In fact, we can replace the polynomial weight by an exponential one, see below.

If $j(r) \Error(x, x_0)$ were in $L^1_x$, then
\be\lb{reason}
\int_{\mc M} \frac 1 {j(d(x_1, x))} \int_{\mc M} \Error(x, x_0) f(x_0) \dd x_0 \dd x \les \bigg(\sup_{x_0 \in \mc M} \int_{\mc M} j(d(x, x_0)) \Error(x, x_0) \dd x\bigg) \int_{\mc M} \frac {f(x_0)} {j(d(x_1, x_0))} \dd x_0 \les \|f\|_\K
\ee
would imply that $\Error \in \B(\K)$. However, since in general $j(r) \Error(x, x_0)$ is not in $L^1_x$, the error term $E$ need not belong to $\U(\K)$.

However, for any $\delta>0$, as $j_\delta(r) \Error(x, x_0) \in L^1_x$, by the same reasoning as (\ref{reason}) $\Error \in \B(\K_\delta)$, where recall $\K_\delta \subset \K$ is defined using $j_\delta$ (\ref{jdelta}). Moreover, now we can state that $\Error, r\Error(x, x_0) \in \B(L^1, \K_\delta)$, with the same proof as (\ref{fir}).

Furthermore, using the explicit form of $\Error$ (\ref{explicit}) and exploiting the $\delta$ difference in exponential decay rates, we can upgrade (\ref{fir}) to
$$
j'_\delta(r) \Error(x, x_0) \in \B(L^1) \cap \B(L^1, \K),
$$
which also implies that $j'_\delta(r) \Error(x, x_0) \in \B(L^1, \tilde \K) \subset \B(\tilde \K)$, as well as $j_\delta(r) \Error \in \B(L^1, \tilde \K)$, since $j_\delta \les j'_\delta$, again with norms of size $O(\epsilon)$, for fixed $\delta$.

As in (\ref{approx2}), for $t  \geq 0$ the approximate sine propagator
$$
S_0(t)(x, x_0) = \frac 1 {4\pi \sqrt {a(x)}} \delta_{d(x_0, x)=t}
$$
is a solution of the equation
$$
\partial^2_t S_0 - (\Delta - \kappa_0) S_0 = E,\ S_0(0)=0,\ \partial_t S_0(0)=\delta_{x_0}(x),
$$
where
$$
E(t)(x, x_0) = \Error(x, x_0) \delta_{d(x_0, x)=t}.
$$
Both $S_0$ and $E$ are supported on the light cone, meaning that $S_0(t)$ and $E(t)$ are supported on $\partial B(x_0, t)$ for each $t \geq 0$.


Recall that the sine propagator is given by the series (\ref{sseries})
$$
S = S_0 + S_0 * E + S_0 * E * E + \ldots = S_0 * (I-E)^{-1},
$$
where $*$ is the algebra operation of $\U$ (\ref{s0e1}). Since the sine propagator $S$ and the approximate propagator $S_0$ are self-adjoint, $S=S^*$, by taking the adjoint one also gets
$$
S = S_0 + E^* \ast S_0 + E^* \ast E^* \ast S_0 + \ldots = (I-E^*)^{-1} * S_0.
$$
Since multiplication by $t$ acts as a derivation,
\be\lb{der}
tS=(tS_0) * (I-E)^{-1} + S_0 * (I-E)^{-1} * (tE) * (I-E)^{-1}.
\ee

We estimate the error $E$ in the $\U(L^1)$ norm: for each $x_0 \in \mc M$
$$
\int_{\mc M} \int_0^\infty |E(t)(x, x_0)| \dd t \dd x = \int_{\mc M} \int_0^\infty |\delta_{d(x_0, x)=t} \Error(x, x_0)| \dd t \dd x = \|\Error(x, x_0)\|_{L^1_x} \les \epsilon.
$$
Thus $E \in \U(L^1)$, with $\|E\|_{\U(L^1)} \les \epsilon$. In the same manner, our prior estimates for $\Error$ immediately imply that
\be\lb{ee1}
E \in \U(L^1) \cap \U(\K_\delta),\ tE(t) \in \U(L^1, \K_\delta),
\ee
as well as
\be\lb{ee2}
j'_\delta(t) E(t) \in \U(L^1) \cap \U(\tilde \K),\ j_\delta(t)E(t) \in \U(L^1, \tilde \K),
\ee
with norms of size $O(\epsilon)$, for fixed $\delta$. More estimates are true, but these suffice for our purposes.

Regarding $S_0$, observe that
\be\lb{s01}
S_0 \in \U(L^1, \K^*) \cap \U(\K, L^\infty),\ tS_0 \in \U(L^1, L^\infty)
\ee
and
\be\lb{s02}
j(t) S_0(t) \in \U(L^1, L^\infty),\ j'(t) S_0(t) \in \U(L^1, \tilde \K^*) \cap \U(\tilde \K, L^\infty).
\ee

The estimates we have selected (\ref{ee1}) imply that
$$
(I-E)^{-1} \in \U(L^1) \cap \U(\K_\delta),\ (I-E)^{-1} * (tE) * (I-E)^{-1} \in \U(L^1, \K_\delta).
$$
Together with (\ref{s01}) and (\ref{der}), this implies that
$$
S \in \U(L^1, \K^*),\ tS \in \U(L^1, L^\infty),
$$
and using the self-adjointness of $S$ we also obtain $S \in \U(\K, L^\infty)$, three estimates that we wanted to prove.

To prove exponential decay, first note that $\sinh(t+s) = \sinh t \cosh s + \cosh t \sinh s$ and
$$
\cosh(t+s) = \cosh t \cosh s + \sinh t \sinh s < 2 \cosh t \cosh s,
$$
meaning that $|j(t+s)| \leq |j(t)| |j'(s)| + |j'(t)| |j(s)|$ and $|j'(t+s)| \leq 2 |j'(t)| |j'(s)|$ and same for $j_\delta$.

Then for $T_1$ and $T_2$ supported on $[0, \infty)$
\be\lb{comp1}
\|j(t) (T_1 * T_2)\|_{\U(L^1, \tilde \K)} \leq \|j(t) T_1\|_{\U(L^1, \tilde \K)} \|j'(t) T_2\|_{\U(L^1)} + \|j'(t) T_1\|_{\U(\tilde \K)} \|j(t) T_2\|_{\U(L^1, \tilde \K)},
\ee
\be\lb{comp2}
\|j(t) (T_1 * T_2)\|_{\U(L^1, L^\infty)} \leq \|j(t) T_1\|_{\U(L^1, L^\infty)} \|j'(t) T_2\|_{\U(L^1)} + \|j'(t) T_1\|_{\U(\tilde \K, L^\infty)} \|j(t) T_2\|_{\U(L^1, \tilde \K)},
\ee
and
\be\lb{comp3}
\|j'(t) (T_1 * T_2)\|_{\U(L^1)} \leq 2 \|j'(t) T_1\|_{\U(L^1)} \|j'(t) T_2\|_{\U(L^1)},\ \|j'(t) (T_1 * T_2)\|_{\U(\tilde K)} \leq 2 \|j'(t) T_1\|_{\U(\tilde K)} \|j'(t) T_2\|_{\U(\tilde K)},
\ee
as well as
\be\lb{comp4}
\|j'(t) (T_1 * T_2)\|_{\U(L^1, \tilde \K^*)} \leq 2 \|j'(t) T_1\|_{\U(L^1, \tilde \K^*)} \|j'(t) T_2\|_{\U(L^1)},\ \|j'(t) (T_1 * T_2)\|_{\U(\tilde K, L^\infty)} \leq 2 \|j'(t) T_1\|_{\U(\tilde K, L^\infty)} \|j'(t) T_2\|_{\U(\tilde K)}.
\ee
The same estimates also hold with $j_\delta$ instead of $j$.

Using (\ref{comp3}), from (\ref{ee2}) we obtain that $j'_\delta(t) (I-E(t))^{-1} \in \U(L^1) \cap \U(\tilde \K^*)$, while using (\ref{comp1}) we obtain that $j_\delta(t)(I-E(t))^{-1} \in \U(L^1, \tilde \K)$, both for sufficiently small $\epsilon$. Combining this with (\ref{s02}) and using (\ref{comp2}) and (\ref{comp4}) leads to
$$
j_\delta(t) S(t) \in \U(L^1, L^\infty),\ j'_\delta(t) S(t) \in \U(L^1, \tilde \K^*) \cap \U(\tilde \K, L^\infty),
$$
which concludes the proof.


%
%
\end{proof}

%

\section{From propagator estimates to decay estimates}

\subsection{Constant positive curvature} This is an important case due to being in some sense the canonical model for conjugate points. To keep things brief, this will be shortest section of the current paper, to be followed by a more detailed exposition in a future paper.

\begin{proof} Normalize the sectional curvature to $1$ and consider the free shifted Laplacian $-\Delta+1$ on $\S^3$, without any added potential ($V=0$). By the results of Proposition \ref{prop1}, the free sine propagator for the shifted wave equation (\ref{wave}) is given for $t>0$ by
$$
S_0(t)(x, x_0) = \frac {\sin(t \sqrt{-\Delta+1})}{\sqrt{-\Delta+1}}(x, x_0) = \frac 1 {4\pi j(t)} \delta_{d(x_0, x)=t},
$$
where $j(t) = \sin t$.

In particular, $S_0$ is periodic of period $2\pi$, so it suffices to study the equation on an interval of period length. Since sine is an odd function, $S_0(t)=-S_0(2\pi-t)$. In particular, $S_0(\pi)=0$. Thus, it is enough to study the evolution on an interval of length $\pi$.

The integral kernel of $S_0(t)$ is a measure for each $x_0$, positive on $(0, \pi)$ and negative on $(\pi, 2\pi)$.

With respect to any choice of origin $x_0 \in \S^3$, one can express integrals in spherical coordinates as
$$
\int_{\S^3} f = \int_0^\pi \int_{\partial B(x_0, r)} f(r, \omega) \dd A = \int_0^\pi \int_{\S^2} f(r, \omega) j^2(r) \dd \omega \dd r.
$$

Let $C_0(t)=\cos(t \sqrt{-\Delta+1})$. It is easy to see that $C_0(\pi)$ is the linear transformation that turns $f(x)$ into $-f(-x)$. The integral kernel of $C_0(t)$ is no longer a measure, but a distribution, given by the explicit formula
\be\lb{c0}
C_0(t)(x, x_0) = \frac {j'(t)}{4\pi j^2(t)} \delta_{d(x_0, x)=t} + \frac 1 {4\pi j(t)} \delta_{d(x_0, x)=t} \partial_r,
\ee
where recall $r=d(x_0, x)$.

By direct computation, one has the fundamental estimates
\be\lb{fund1}
\int_0^\pi |S_0(t)(x, x_0)| \dd t = \frac 1 {4\pi j(d(x, x_0))} \les \frac 1 {d(x, x_0)} + \frac 1 {d(x, -x_0)},\ \int_0^\pi j(t) |S_0(t)(x, x_0)| \dd t = \frac 1 {4\pi},
\ee
as well as
\be\lb{fund2}
\int_0^\pi |C_0(t) f|(x) \dd t \leq \frac 1 {2\pi} \int_{\S^3} \frac {|\nabla f(x_0)| \dd x_0}{j(d(x, x_0))} + \frac 1 {2\pi} \int_{\S^2} |f| \les \int_{\S^2} \frac {|\nabla f(x_0)|}{d(x, x_0)} + \frac {|\nabla f(x_0)|}{d(x, -x_0)} + |f(x_0)| \dd x_0,
\ee
where $\S^2$ is the equator between $x$ and $-x$,
$$
\frac 1 {4\pi} \int_{\S^2} |f| = [\frac {\sin(\pi/2)\sqrt{-\Delta+1}}{\sqrt{-\Delta+1}}f](x) \les \int_{\S^3} |\nabla f(x_0)| + |f(x_0)| \dd x_0,
$$
and finally
\be\lb{fund3}
\int_0^\pi j(t) |C_0(t) f|(x) \dd t \les \int_{\S^3} |\nabla f(x_0)| + |f(x_0)| \dd x_0.
\ee

From (\ref{fund1}), for $t \in (0, \pi/2]$
$$
2\bigg|\frac {\cos(t \sqrt{-\Delta+1})}{-\Delta+1}(x, x_0)\bigg| \leq \int_{t}^{\pi-t} |S_0(\tau)(x, x_0)| \dd \tau \leq \frac 1 {4\pi j(t)},
$$
so for $t \in (0, \pi)$
\be\lb{next1}
\bigg\|\frac {\cos(t \sqrt{-\Delta+1})}{-\Delta+1} f\bigg\|_{L^\infty} \les \frac 1 {\sin t} \|f\|_{L^1},\ \|\cos(t \sqrt{-\Delta+1}) f\|_{L^\infty} \les \frac 1 {\sin t} \|(-\Delta+1)f\|_{L^1} \les \frac 1 {\sin t} \|f\|_{W^{2, 1}}.
\ee
In the same way, from (\ref{fund3}), for $t \in (0, \pi)$
\be\lb{next2}
\bigg\|\frac {\sin(t \sqrt{-\Delta+1})}{\sqrt{-\Delta+1}} f\bigg\|_{L^\infty} \les \frac 1 {\sin t} \int_{\S^3} |\nabla f(x_0)| + |f(x_0)| \dd x_0 = \frac 1 {\sin t} \|f\|_{W^{1, 1}}.
\ee
From (\ref{next1}) and (\ref{next2}) we infer the usual $L^p \to L^{p'}$ decay estimates: for $p \in (1, 2]$ and $t \in (0, \pi)$
$$
\bigg\|\frac {e^{it\sqrt{-\Delta+1}}}{(-\Delta+1)^{\frac 1 {p} - \frac 1 {p'}}} f\bigg\|_{L^{p'}} \les (\sin t)^{\frac 1 {p'} - \frac 1 p} \|f\|_{L^p}.
$$
These imply all the usual Strichartz estimates for the wave equation (\ref{wave}) on $\S^3$.

As already mentioned, an alternative point of view is that the shifted Laplacian $-\Delta+1$ has discrete spectrum consisting of $\{(\ell+1)^2 \mid \ell \in \Z, \ell \geq 0\}$, where each eigenvalue $(\ell+1)^2$ has multiplicity $(\ell+1)^2$. Due to the high multiplicity of the eigenvalues, the sphere can be considered a special case, the ``worst-case scenario'' for decay, by comparison with a generic compact manifold of the same dimension.

Let $P_\ell$ be the $L^2$ spectral projection corresponding to the eigenvalue $(\ell+1)^2$. Then the wave propagators are given by
$$
S_0(t) = \sum_{\ell=0}^\infty \frac {\sin(t(\ell+1))}{\ell+1} P_\ell,\ C_0(t) = \sum_{\ell=0}^\infty \cos(t(\ell+1)) P_\ell.
$$
Conversely, $P_\ell$ can be expressed in terms of the free sine propagator:
$$
P_\ell = \frac {\ell+1} \pi \int_0^{2\pi} S_0(t) \sin(t(\ell+1)) \dd t.
$$

This leads to the following bounds for $P_\ell$:
$$
\sup_{\ell} \frac 1 {\ell+1} P_\ell(x, x_0) \les \frac 1 {d(x, x_0)} + \frac 1 {d(x, -x_0)},\ \sup_{x, x_0} \sup_\ell \frac 1 {(\ell+1)^2} P_\ell(x, x_0) < \infty,
$$
which imply $\K \to L^\infty$, $L^1 \to \K^*$, and various $L^p \to L^q$ bounds.

We are also interested in dyadic projections such as $\ds \tilde P_k = \sum_{\ell=2^k-1}^{2^{k+1}-2} P_\ell$. In terms of the sine propagator,
$$
\tilde P_k = \int_0^{2\pi} S(t) F(t, k) \dd t
$$
where $F$ is an expression such that $\ds |F(t, k)| \les \frac {2^k}{\sin t}$. Then by (\ref{fund1})
$$
\sup_k \frac 1 {2^k} \tilde P_k \les \frac 1 {d^2(x, x_0)} + \frac 1 {d^2(x, -x_0)}.
$$
\end{proof}

\subsection{Constant negative curvature} \begin{proof} Normalize the sectional curvature to $-1$ and consider the free shifted Laplacian $-\Delta-1$ on $\H^3$, to which we shall later add a small potential $V$. By Proposition \ref{prop1}, the sine propagator for (\ref{wave}) is given for $t>0$ by
$$
S_0(t)(x, x_0) = \frac {\sin(t \sqrt{-\Delta-1})}{\sqrt{-\Delta-1}}(x, x_0) = \frac 1 {4\pi j(t)} \delta_{d(x_0, x)=t},
$$
where now $j(t)=\sinh t$.

The following fundamental estimates hold for the free propagators:
\be\lb{fund4}
\int_0^\infty |S_0(t)(x, x_0)| \dd t = \frac 1 {4\pi j(r)},\ \int_0^\infty t |S_0(t)(x, x_0)| \dd t \leq \int_0^\infty j(t) |S_0(t)(x, x_0)| \dd t = \frac 1 {4\pi},
\ee
where $r=d(x_0, x)$. The integral kernel of the cosine propagator $C_0(t)$ is again given by (\ref{c0})
$$
C_0(t)(x, x_0) = \frac {j'(t)}{4\pi j^2(t)} \delta_{d(x_0, x)=t} + \frac 1 {4\pi j(t)} \delta_{d(x_0, x)=t} \partial_r,
$$
where now $j(t)=\sinh t$. It fulfills the estimates
\be\lb{fund5}
\int_0^\infty |C_0(t) f|(x) \dd t \leq \frac 1 {2\pi} \int_{\H^3} \frac {|\nabla f(x_0)| \dd x_0}{j(r)}
\ee
and
\be\lb{fund6}
\int_0^\infty t |C_0(t) f|(x) \dd t \leq \int_0^\infty j(t) |C_0(t) f|(x) \dd t \les \int_{\H^3} |\nabla f(x_0)| \dd x_0.
\ee
Hence we infer the exponential $L^p$ decay estimates
\be\lb{exp1}
\|S_0(t) f\|_{L^\infty} \les |j(t)|^{-1} \|f\|_{\dot W^{1, 1}},\ \|C_0(t) f\|_{L^\infty} \les |j(t)|^{-1} \|(-\Delta-1)f\|_{L^1} \les |j(t)|^{-1} \|f\|_{W^{2, 1}},
\ee
and
\be\lb{exp2}
\bigg\|\frac {e^{it\sqrt{-\Delta-1}}}{(-\Delta-1)^{\frac 1 {p} - \frac 1 {p'}}} f\bigg\|_{L^{p'}} \les |j(t)|^{\frac 1 {p'} - \frac 1 p} \|f\|_{L^p}.
\ee
In turn, these imply all the usual Strichartz estimates for the shifted wave equation (\ref{wave}).

Next, consider Schr\"{o}dinger's equation (\ref{sch}). By functional calculus, the solution is given by
\be\lb{long}\begin{aligned}
	e^{itH_0} f &= \frac 1 {2\pi i} \int_0^\infty e^{it\lambda} [R_0(\lambda-i0) - R_0(\lambda+i0)] \dd \lambda \\
	&= \frac i {\pi} \int_{-\infty}^\infty e^{it\lambda^2} R_0((\lambda+i0)^2) \lambda \dd \lambda \\
	&= C t^{-1} \int_{-\infty}^\infty e^{it\lambda^2} \partial_\lambda [R_0((\lambda-i0)^2)] \dd \lambda \\
	&=\tilde C t^{-1} \int_{-\infty}^\infty \mc F[e^{it\lambda^2}] \mc F \{\partial_\lambda [R_0((\lambda-i0)^2)]\} \dd \lambda
\end{aligned}\ee
where $R_0$ is the free resolvent, $R_0(\lambda)=(-\Delta-1-\lambda)^{-1}$, and $\mc F$ is the Fourier transform. The first Fourier transform is of size $\mc F[e^{it\lambda^2}] \les t^{-1/2}$, while the second is given by
\be\lb{fou}
\mc F \{R_0((\lambda+i0)^2)\} = \chi_{t \geq 0} S_0(t),\ \mc F \{\partial_\lambda [R_0((\lambda+i0)^2)]\} = \chi_{t \geq 0} it S_0(t). 
\ee
The latter expression is integrable by (\ref{fund4}), producing the desired Schr\"{o}dinger decay estimate
\be\lb{decay}
\|e^{itH} f\|_{L^\infty} \les |t|^{-3/2} \|f\|_{L^1}.
\ee
In turn, this implies all the usual Strichartz estimates for Schr\"{o}dinger's equation (\ref{sch}).
\end{proof}

\subsection{Adding a small scalar potential} The above constant negative curvature estimates work in the same manner in the presence of a small scalar potential $V \in \K$, after replacing $R_0$ by the perturbed resolvent $R(\lambda)=(-\Delta-1+V-\lambda)$ and the free propagators $S_0(t)$ and $C_0(t)$ by the perturbed propagators
$$
S(t) = \frac {\sin(t \sqrt H)}{\sqrt H} = \frac {\sin(t \sqrt {-\Delta-1+V})}{\sqrt {-\Delta-1+V}},\ C(t) = \cos(t \sqrt H) = \cos(t \sqrt{-\Delta-1+V}).
$$
\begin{proof}
By the resolvent identity, for small $V$ the perturbed resolvent can be expanded into a Born series
$$
R=R_0(I+VR_0)^{-1}=R_0-R_0VR_0+R_0VR_0VR_0-\ldots.
$$
Equivalently, after taking the Fourier transform as in (\ref{fou}), by Duhamel's formula the perturbed propagator can be expressed as follows
\be\lb{duhamel}\begin{aligned}
\chi_{t \geq 0} S &= [\chi_{t \geq 0} S_0(t)] * (I + [V\chi_{t \geq 0} S_0(t)])^{-1} \\
&= \chi_{t \geq 0} S_0 - [\chi_{t \geq 0} S_0] * [V\chi_{t \geq 0}S_0] + [\chi_{t \geq 0} S_0] * [V\chi_{t \geq 0}S_0] * [V\chi_{t \geq 0}S_0] - \ldots.
\end{aligned}\ee
where the algebra operation $*$ is defined by (\ref{s0e1}); also see (\ref{s0e}).

Recall the algebroid $\U$ (\ref{u}). Since
$$
\|V \chi_{t \geq 0} S_0(t)\|_{\U(L^1) \cap \U(\K)} \les \|V\|_\K,\ \|\chi_{t \geq 0} S_0(t) V\|_{\U(L^\infty) \cap \U(\K^*)} \les \|V\|_\K,
$$
when $\|V\|_\K$ is sufficiently small the infinite series that equals $(I+[V \chi_{t \geq 0} S_0])^{-1}$ converges and
\be\lb{auxx}
(I+[V \chi_{t \geq 0} S_0])^{-1} \in \U(L^1) \cap \U(\K).
\ee
The convergence of this series and of (\ref{duhamel}) also guarantees the absence of eigenvalues, embedded or otherwise, and of threshold resonances.

Since $\chi_{t \geq 0} S_0 \in \U(L^1, \K^*) \cap \U(\K, L^\infty)$, by (\ref{duhamel}) and (\ref{auxx})
\be\lb{a1}
\chi_{t \geq 0} S(t) \in \U(L^1, \K^*) \cap \U(\K, L^\infty)
\ee
as well and, since multiplying by $t$ is a derivation on $\U$,
\be\lb{a2}
\chi_{t \geq 0} tS = [\chi_{t \geq 0} tS_0] * (I+[V \chi_{t \geq 0} S_0])^{-1} - S_0 * (I+[V \chi_{t \geq 0} S_0])^{-1} * (V[\chi_{t \geq 0}tS_0]) * (I+[V \chi_{t \geq 0} S_0])^{-1} \in \U(L^1, L^\infty).
\ee
Hence (\ref{decay}) also holds in the presence of a small potential.

Regarding the cosine propagator $C(t)=\cos(t \sqrt H)$, taking a $t$ derivative in (\ref{duhamel}) produces
$$\begin{aligned}
\chi_{t \geq 0} C &= (I + [\chi_{t \geq 0} S_0(t) V])^{-1} * [\chi_{t \geq 0} C_0(t)] \\
&= \chi_{t \geq 0} S_0 - [\chi_{t \geq 0} S_0 V] * [\chi_{t \geq 0}C_0] + [\chi_{t \geq 0} S_0 V] * [\chi_{t \geq 0}S_0 V] * [\chi_{t \geq 0}C_0] - \ldots.
\end{aligned}$$

As a consequence of (\ref{fund5}), (\ref{fund6}), and (\ref{auxx}), estimates analogous to (\ref{a1}) and (\ref{a2}) hold for the perturbed cosine propagator $C(t)$:
$$
f \mapsto \int_0^\infty |C(t) f| \dd t \in \B(\dot W^{1, 1}, \K^*) \cap \B(\nabla^{-1} \K, L^\infty),\ f \mapsto \int_0^\infty |t C(t) f| \in \B(\dot W^{1, 1}, L^\infty).
$$

Estimates (\ref{a1}) and (\ref{a2}), together with the analogous estimates for the cosine propagator, lead to the polynomial decay bounds
$$
\|S(t)f\|_{L^\infty} \les |t|^{-1} \|f\|_{\dot W^{1, 1}},
$$
\be\lb{cos_poly}
\|C(t) f\|_{L^\infty} \les |t|^{-1} \|(-\Delta-1)f\|_{L^1} \les |t|^{-1} \|f\|_{W^{2, 1}}
\ee
and
$$
\bigg\| \frac {e^{it\sqrt{-\Delta+1}}}{(-\Delta+1)^{\frac 1 p - \frac 1 {p'}}} \bigg\|_{L^{p'}} \les |t|^{\frac 1 {p'} - \frac 1 p} \|f\|_{L^p}.
$$

In order to obtain exponential decay bounds, suppose $V \in \tilde \K$, the modified Kato space $\tilde \K = \K \cap L^1$ (\ref{modk}), which can also be defined as
$$
\tilde \K = \{V \mid \sup_{x_0 \in \H^3} \int_{\H^3} \frac {|V(x)| j'(d(x_0, x))}{j(d(x_0, x))} \dd x < \infty\}.
$$
Again note that $L^1 \cap \K = \tilde K$, where $\K$ improves local regularity.

Then
$$
\sup_{x_0} \bigg\|\int_0^\infty j'(t) |V(x)| |S_0(t)(x, x_0)| \dd t\bigg\|_{L^1_x} \les \sup_{x_0} \int_{\H^3} \frac {|V(x)| j'(d(x_0, x))}{j(d(x_0, x))} \dd x \les \|V\|_{\tilde \K}
$$
so $\chi_{t \geq 0} j'(t) V S_0(t) \in \U(L^1)$. Likewise, $\chi_{t \geq 0} j'(t) V S_0(t) \in \U(\tilde \K)$ because
$$
\int_{\H^3} \frac {|V(x)| j'(d(x_0, x))}{j(d(x_0, x))} f(x_0) \dd x_0 \les |V(x)| \|f\|_{\tilde \K}.
$$
Moreover
$$
\sup_{x_0} \bigg\|\int_0^\infty j(t) |V(x)| |S_0(t)(x, x_0)| \dd t\bigg\|_{\tilde \K} \les \|V\|_{\tilde \K},
$$
meaning that $\chi_{t \geq 0} j(t) V S_0 \in \U(L^1, \tilde \K)$.

Regarding $S_0$, observe again that
\be\lb{props}
j(t) S_0(t) \in \U(L^1, L^\infty),\ j'(t) S_0(t) \in \U(L^1, \tilde \K^*) \cap \U(\tilde \K, L^\infty).
\ee

We now use these inequalities when summing the series
$$
(I + [V \chi_{t \geq 0} S_0(t)])^{-1} = I - V \chi_{t \geq 0} S_0(t) + [V \chi_{t \geq 0} S_0(t)] * [V \chi_{t \geq 0} S_0(t)] - \ldots
$$
The bounds (\ref{comp3}) lead to
$$
j'(t) (I + \chi_{t \geq 0} V S_0(t))^{-1} \in \U(L^1) \cap \U(\tilde \K)
$$
if $\|V\|_{\tilde \K}$ is sufficiently small. The same and (\ref{comp1}) produce
$$
j(t) (I + \chi_{t \geq 0} V S_0(t))^{-1} \in \U(L^1, \tilde \K).
$$
Finally, using (\ref{comp2}), (\ref{props}), and (\ref{comp4}) as well, we can sum the Duhamel expansion (\ref{duhamel}) and obtain that
\be\lb{exp_decay}
\chi_{t \geq 0} j'(t) S \in \U(L^1, \tilde \K^*) \cap \U(\tilde \K, L^\infty),\ \chi_{t \geq 0} j(t) S \in \U(L^1, L^\infty).
\ee
The analogous estimates for the cosine propagator are also true and proved in the same manner.

Thus the exponential decay estimates (\ref{exp1}) and (\ref{exp2}) are preserved after adding a small potential in $\tilde \K$.
\end{proof}

\subsection{Non-constant negative sectional curvature with a potential}
\begin{proof}[Proof of Theorem \ref{thm1}] Let $S$ be the free sine propagator. In Proposition \ref{prop3} we proved that $tS(t) \in \U(L^1, L^\infty)$, $S \in \U(L^1, \K^*) \cap \U(\K, L^\infty)$, as well as that $S$ has the exponential decay properties analogous to (\ref{props})
$$
j_\delta(t) S(t) \in \U(L^1, L^\infty),\ j'_\delta(t) S(t) \in \U(L^1, \tilde \K^*) \cap \U(\tilde \K, L^\infty).
$$
	
Let $S_V$ be the perturbed sine propagator. If $V \in \K$ is small in norm, in the same manner as in (\ref{a1}) we get that $S \in \U(L^1, \K^*) \cap \U(\K, L^\infty)$ and in the same manner as in (\ref{a2}) we get that $tS_V(t) \in \U(L^1, L^\infty)$. The latter implies the $t^{-1}$ $L^1 \to L^\infty$ decay estimate for the cosine propagator, same as (\ref{cos_poly}).

The $t^{-3/2}$ decay for Schr\"{o}dinger's equation also follows from $tS_V(t) \in \U(L^1, L^\infty)$, by the same computation as in (\ref{long}).

If in addition $V \in \tilde \K$ and $V$ is small in norm, then using the same algebra structure (\ref{comp1}--\ref{comp4}) leads to the analogue of (\ref{exp_decay})
$$
j_\delta(t) S_V(t) \in \U(L^1, L^\infty),\ j'_\delta(t) S_V(t) \in \U(L^1, \tilde \K^*) \cap \U(\tilde \K, L^\infty).
$$
In particular, for $t>0$
$$
\bigg\|\int_t^\infty |S_V(t)(x, y)| \dd t \bigg\|_{\B(L^1, L^\infty)} \les j_\delta(t)^{-1}.
$$
This proves the exponential decay with a loss of the perturbed cosine propagator. 
\end{proof}

\section*{Acknowledgments}
I would like to thank Gong Chen, Wilhelm Schlag, and Robert Strichartz for the useful conversations. I would also like to thank Wilhelm Schlag for his hospitality during my visits at the University of Chicago when this paper was conceived.


\begin{thebibliography}{AbcDef1}
\bibitem[AnkPie1]{AnkPie1} J.~Anker, V.~Pierfelice, \emph{Nonlinear Schr\"{o}dinger equation on real hyperbolic spaces}, Ann.~Inst.~H.~Poincar\'{e} (C) Non Linear Analysis 26 (2009), pp.~1853--1869.
\bibitem[AnkPie2]{AnkPie2} J.~Anker, V.~Pierfelice, \emph{Wave and Klein–Gordon equations on hyperbolic spaces}, Anal.~PDE,
7(2014), pp.~953-995.
\bibitem[AnPiVa1]{AnPiVa1} J.~Anker, V.~Pierfelice, M.~Vallarino, \emph{The wave equation on hyperbolic spaces}, Journal of Differential Equations, 2012, 252 (10), pp.~5613--5661.
\bibitem[AnPiVa2]{AnPiVa2} J.~Anker, V.~Pierfelice, M.~Vallarino, \emph{The wave equation on Damek--Ricci spaces}, Ann.~Mat.~Pura Appl.~(4), 194(3) pp.~731--758 (2015).
\bibitem[BecGol1]{BecGol1} M.~Beceanu, M.~Goldberg \emph{Dispersive estimates for the Schrödinger equation with scaling-critical potential}, Communications in Mathematical Physics (2012), Vol.~314, Issue 2, pp.~471--481.
\bibitem[BecGol2]{BecGol2} M.~Beceanu, M.~Goldberg, \emph{Strichartz estimates and maximal operators for the wave equation in $\R^3$}, Journal of Functional Analysis (2014) 266, Issue 3, pp.~147--151.
\bibitem[BecKwo]{BecKwo} M.~Beceanu, H.~Kwon, \emph{Decay estimates for Schr\"{o}dinger's equation with magnetic potentials in three dimensions}, preprint, arXiv:2411.11787.
\bibitem[BerLof]{BerLof} J.~Bergh, J.~L\"{o}fstr\"{o}m, \emph{Interpolation spaces}, Springer--Verlag, 1976.
\bibitem[Cal]{Cal} D.~Calegari, \emph{Introduction to differential geometry}, preprint, online.
\bibitem[Che]{Che} X.~Chen, \emph{Resolvent and spectral measure on non-trapping asymptotically hyperbolic manifolds III. Global-in-time Strichartz estimates without loss}, Ann. I.~H.~Poincar\'{e} -AN 35 (2018), pp.~803--829.
\bibitem[Doc]{Doc} M.~Do Carmo, \emph{Riemannian Geometry}, Birkh\"{a}user, 1992.
\bibitem[GebTat]{GebTat} D.~Geba, D.~Tataru, \emph{Dispersive estimates for wave equations}, Comm.~Partial Differential Equations 30 (2005), no.~4--6, pp.~849--880.
\bibitem[Has]{Has} A.~Hassani, \emph{Wave equation on Riemannian symmetric spaces}, J.~Math.~Phys.~52 (2011), 043514.
\bibitem[HasZha]{HasZha} A.~Hassell, J.~Zhang, \emph{Global-in-time Strichartz estimates on nontrapping asymptotically conic manifolds}, Analysis \& PDE, 9 (2016), pp.~151--192.
\bibitem[Ion]{Ion} A.~Ionescu, \emph{Fourier integral operators on noncompact symmetric spaces of real rank one}, J.~Funct.~Anal.~174 (2000), pp.~274--300.
\bibitem[Kap]{Kap} L.~Kapitanski, \emph{Norm estimates in Besov and Lizorkin--Treibel spaces for the solutions of second order linear hyperbolic equations}, J.~Sov.~Math., 56 (1991), pp.~2348--2389.
\bibitem[Lee]{Lee} J.~Lee, \emph{Riemannian manifolds. An introduction to curvature}, Springer, 1997.
\bibitem[MaMeTa1]{MaMeTa1} J.~Marzuola, J.~Metcalfe, D.~Tataru, \emph{Wave packet parametrices for evolutions governed by pdo's with rough symbols}, Proc.~Amer.~Math.~Soc.~136 (2008), no.~2, pp.~597--604.
\bibitem[MaMeTa2]{MaMeTa} J.~Marzuola, J.~Metcalfe, D.~Tataru, \emph{Strichartz estimates and local smoothing estimates for asymptotically flat Schr\"{o}dinger equations}, J.~Funct.~Anal.~255 (2008), no.~6, pp.~1497--1553.
\bibitem[MetTay]{MetTay} J.~Metcalfe, M.~Taylor, \emph{Nonlinear waves on 3D hyperbolic space}, Transactions of the American Mathematical Society
Vol.~363, no.~7, pp.~3489--3529.
\bibitem[MetTat]{MetTat} J.~Metcalfe, D.~Tataru, \emph{Global parametrices and dispersive estimates for variable coefficients wave equations}, Math.~Ann.~353 (2012), no. 4, pp.~1183--1237.
\bibitem[MoSeSo]{MoSeSo} G.~Mockenhaupt, A.~Seeger, C.~Sogge, \emph{Local smoothing of Fourier integral operators and Carleson--Sj\"{o}lin estimates}, J.~Amer. Math.~Soc., 6 (1993), pp.~65--130.
\bibitem[RodSch]{RodSch} I.~Rodnianski, W.~Schlag, \emph{Time decay for solutions of Schr\"{o}dinger equations with rough and time-dependent potentials}, Invent.~Math.~155 (2004), no.~3, pp.~451--513.
\bibitem[ScSoSt1]{ScSoSt1} W.~Schlag, A.~Soffer, W.~Staubach, \emph{Decay for the wave and Schrödinger evolutions on manifolds with conical ends, Part I}, Trans.~Amer.~Math.~Soc.~362 (2010), no.~1, pp.~19-52.
\bibitem[ScSoSt2]{ScSoSt2} W.~Schlag, A.~Soffer, W.~Staubach, \emph{Decay for the wave and Schrödinger evolutions on manifolds with conical ends, Part II}, Trans.~Amer.~Math.~Soc.~362 (2010), no.~1, pp.~289-318.
\bibitem[SSWZ]{SSWZ} Y.~Sire, C.~Sogge, C.~Wang, J.~Zhang, \emph{Strichartz estimates and Strauss conjecture on non-trapping asymptotically hyperbolic manifolds}, Transactions of the American Mathematical Society, Vol.~373, No.~11, 2020, pp.~7639--7668.
\bibitem[Smi]{Smi} H.~Smith, \emph{A parametrix construction for wave equations with $C^{1,1}$ coefficients}, Ann.~Inst.~Fourier (Grenoble), 48(3), pp.~797–835, 1998.
\bibitem[SmiSog1]{SmiSog1} H.~Smith, C.~Sogge, \emph{On Strichartz and eigenfunction estimates for low regularity metrics}, Math.~Res.~Lett., 1(6) pp.~729--737, 1994.
\bibitem[SmiSog2]{SmiSog2} H.~Smith, C.~Sogge, \emph{Global Strichartz estimates for nontrapping perturbations of the Laplacian}, Comm.~Partial Differential Equations, 25(11-12), pp.~2171–2183, 2000.
\bibitem[SogWan]{SogWan} C.~Sogge, C.~Wang, \emph{Concerning the wave equation on asymptotically Euclidean manifolds}, Journal d'Analyze Math\'{e}matique, Vol.~112, pp.~1--32, 2010.
\bibitem[StaTat]{StaTat} G.~Staffilani, D.~Tataru, \emph{Strichartz estimates for a Schr\"{o}dinger equation with nonsmooth coefficients}, Comm.~Partial Differential Equations 27 (2002), pp.~1337--1372.
\bibitem[Tat1]{Tat1} D.~Tataru, \emph{Strichartz estimates for operators with nonsmooth coefficients and the nonlinear wave equation} Amer.~J.~Math., 122(2), pp.~349--376, 2000.
\bibitem[Tat2]{Tat2} D.~Tataru, \emph{Strichartz estimates for second order hyperbolic operators with nonsmooth coefficients II}, Amer.~J.~Math., 123(3) pp.~385--423, 2001.
\bibitem[Tat3]{Tat3} D.~Tataru, \emph{Strichartz estimates in the hyperbolic space and global existence for the semilinear wave equation}, Trans.~Amer.~Math.~Soc.~353 (2001), pp.~795--807.
\bibitem[Tat4]{Tat4} D.~Tataru, \emph{Strichartz estimates for second order hyperbolic operators with nonsmooth coefficients III}. J.~Amer.~Math.~Soc., 15 (2) pp.~419--442, 2002.
\bibitem[Tat5]{Tat5} D.~Tataru, \emph{Parametrices and dispersive estimates for Schr\"{o}dinger operators with variable coefficients}, Amer.~J.~Math. 130 (2008), no.~3, pp.~571--634.
\bibitem[Tat6]{Tat6} D.~Tataru, \emph{Local decay of waves on asymptotically flat stationary space-times}, Amer.~J.~Math.~135 (2013), no.~2, pp.~361--401.
\bibitem[Zha]{Zha} J.~Zhang, \emph{Strichartz estimates and nonlinear wave equation on nontrapping asymptotically conic manifolds}, Advances in Math., 271(2015), pp.~91--111.
\end{thebibliography}
\end{document}